\documentclass[11pt,a4paper]{article}

\usepackage[left=2.45cm, top=2.45cm,bottom=2.45cm,right=2.45cm]{geometry}

\usepackage[colorlinks = TRUE]{hyperref}

\usepackage{mathtools,amssymb,amsthm,mathrsfs,calc,graphicx,stmaryrd,xcolor,cleveref,dsfont,mathpazo,tikz,pgfplots}
\usepackage[british]{babel}
\usepackage{amsfonts}              % for blackboard bold, etc

\numberwithin{equation}{section}
\numberwithin{figure}{section}

\newtheorem {theorem}{Theorem}[section]

\newtheorem {lemma}[theorem]{Lemma}
\newtheorem {corollary}[theorem]{Corollary}

{\theoremstyle{definition}

}

\def\ba{\begin{array}}
\def\ea{\end{array}}
\def\bea{\begin{eqnarray} \label}
\def\eea{\end{eqnarray}}
\def\be{\begin{equation} \label}
\def\ee{\end{equation}}
\def\bit{\begin{itemize}}
\def\eit{\end{itemize}}
\def\ben{\begin{enumerate}}
\def\een{\end{enumerate}}

\newcommand\N{\mathbb{N}}

\newcommand\R{\mathbb{R}}

\DeclareMathOperator{\myspan}{span}

\DeclareMathOperator{\lin}{lin}

\usepackage{fixme}
\fxsetup{status=draft, theme = color}
\definecolor{fxtarget}{rgb}{0.8000,0.0000,0.0000}

\makeatletter
\let\@fnsymbol\@alph
\makeatother

\usepackage{float}

\begin{document}

\title{\bfseries A Blaschke--Petkantschin formula for linear and affine subspaces with application to intersection probabilities
}

\author{Emil Dare\footnotemark[1]\;, Markus Kiderlen\footnotemark[2]\; and Christoph Th\"ale\footnotemark[3]}

\date{}
\renewcommand{\thefootnote}{\fnsymbol{footnote}}
\footnotetext[1]{Aarhus University, Department of Mathematics, Aarhus C, Denmark. Email: dare@math.au.dk}

\footnotetext[2]{Aarhus University, Department of Mathematics, Aarhus C, Denmark. Email: kiderlen@math.au.dk}

\footnotetext[3]{
Ruhr University Bochum, Faculty of Mathematics, Bochum, Germany. Email: christoph.thaele@rub.de}

\maketitle

\begin{abstract}
\noindent Consider a uniformly  distributed random linear subspace $L$ and a stochastically independent random affine subspace $E$ in $\R^n$, both of fixed dimension. For a natural class of distributions for $E$ we show that the intersection $L\cap E$ admits a density with respect to the invariant measure. This density depends only on the distance $d(o,E \cap L)$ of $L\cap E$ to the origin and is derived explicitly. It can be written as the product of a power of $d(o,E \cap L)$ and a part involving an incomplete beta integral. Choosing $E$ uniformly among all affine subspaces of fixed dimension hitting the unit ball, we derive an explicit density for the random variable $d(o,E \cap L)$ and study the behavior of the probability that $E \cap L$ hits the unit ball in high dimensions. Lastly, we show that our result can be extended to the setting where $E$ is tangent to the unit sphere, in which case we again derive the density for $d(o,E \cap L)$. Our probabilistic results are derived by means of a new integral-geometric transformation formula of Blaschke--Petkantschin type.
\end{abstract}

\section{Introduction and motivation}

Fix dimension parameters $n\geq 2$, $q\in\{1,\ldots,n-1\}$ and $\gamma\in \{0,\ldots,q-1\}$. Let $L_1$ be a $q$-dimensional random linear subspace and $L_2$ be an $(n-q+\gamma)$-dimensional random linear subspace of $\R^n$. We assume that both subspaces are stochastically independent and that $L_1$ and $L_2$ are selected according to the uniform distribution on the Grassmannian $G(n,k)$ of all $k$-dimensional linear subspaces of $\mathbb{R}^n$, with $k=q$ and $k=n-q+\gamma$, respectively. In other words, we use the normalized rotation invariant measures $\nu_k$, $k=q$ and $k=n-q+\gamma$, on these spaces as our underlying probability measures; these and further concepts will formally be introduced in Section \ref{sec:Notation}. The intersection $L_1\cap L_2$ is almost surely a random subspace of $\mathbb{R}^n$ of dimension $\gamma$ and its distribution is known to be the uniform distribution on the space $G(n,\gamma)$, see Figure \ref{3Situations}.

\begin{figure}[t]
	\centering
	\includegraphics[width=\columnwidth]{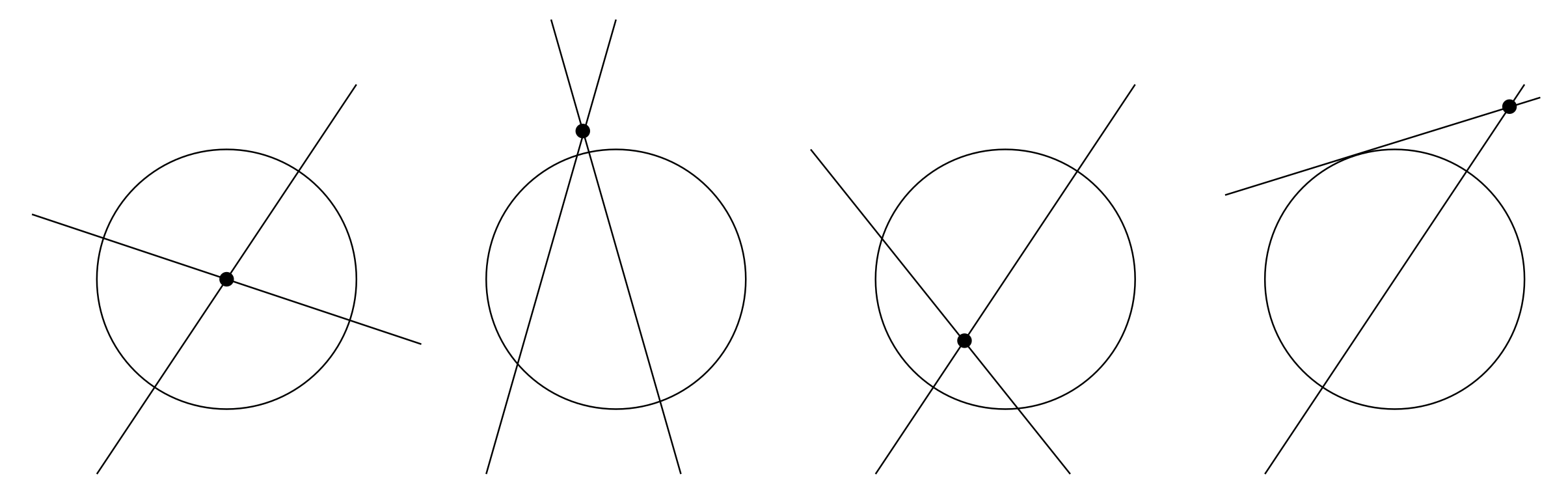}
	\caption{Illustration in the case  $n=2$, $q=1$, $\gamma=0$. From left to right: Intersection of two linear subspaces; intersection of two affine subspaces hitting the unit ball; intersection of a linear with an affine subspace hitting the unit ball; intersection of a linear subspace with an affine subspace tangent to the unit sphere.}
	\label{3Situations}
\end{figure}
\medskip 

Let us now change the set-up and let $E_1$ be a $q$-dimensional random affine subspace and $E_2$ be another $(n-q+\gamma)$-dimensional random affine subspace of $\mathbb{R}^n$. 
Since the motion invariant measure $\mu_k$ on $A(n,k)$, $k\in \{0,\ldots,n-1\}$, is not finite, we restrict attention to the set
$$
[B^n]_k=\{E\in A(n,k): E\cap B^n\ne \emptyset\}
$$
of $k$-dimensional affine subspaces hitting the unit ball $B^n$ of $\R^n$. We thus take
$E_1$ and $E_2$ as random affine subspaces distributed according to the normalized measures $\mu_q$ and $\mu_{n-q+\gamma}$ restricted to $[B^n]_q$ and $[B^n]_{n-q+\gamma}$, respectively. 
Assuming that $E_1$ and $E_2$ are stochastically independent, the intersection $E_1\cap E_2$ is almost surely a random affine subspace of $\mathbb{R}^n$ with dimension $\gamma$. However, the intersection of $E_1\cap E_2$ with the unit ball $B^n$ may  or may not be empty with strictly positive probability, see Figure \ref{3Situations}. Already this basic observation shows that -- in contrast to the case of random linear subspace discussed above -- the distribution of $E_1\cap E_2$ cannot coincide with the normalized motion-invariant measure on $[B^n]_\gamma$. More precisely, since the distance of the intersection of $E_1$ and $E_2$ to the origin $o\in\R^n$ can be arbitrarily large, the distribution of $E_1\cap E_2$ cannot even be supported on a compact subset of $A(n,\gamma)$. In fact, it is known from \cite[Thm.~7.2.8]{SchWeil2008} that the distribution of $E_1\cap E_2$ has a non-trivial density with respect to the invariant measure on $A(n,\gamma)$, which is proportional to
$$
E\mapsto \int_{A(E,q)\cap[B^n]_{q}}\int_{A(E,n-q+\gamma)\cap[B^n]_{n-q+\gamma}}[E_1,E_2]^{\gamma+1}\,\mu_{n-q+\gamma}^E(dE_2)\mu_{q}^E(dE_1),
$$
where $A(E,q)$ is the set of $q$-dimensional affine subspaces containing $E$ and $\mu_{q}^E$ is the invariant measure on that space (similarly for $A(E,n-q+\gamma)$ and $\mu_{n-q+\gamma}^E$). Moreover, $[E_1,E_2]$ stands for the so-called
subspace determinant, describing the relative position of $E_1$ and $E_2$, see below for a detailed definition. We remark that in the special case $n=2$, $q=1$ and $\gamma=0$ this is a classical result of M.\ Crofton discussed in \cite[\S 7]{BlaschkeIntegralgeometrie}, whereas the case $n=3$, $q=2$ and $\gamma=1$ goes back to W.\ Blaschke \cite[\S 33]{BlaschkeIntegralgeometrie}.

On a more abstract level, both problems just mentioned naturally lead to the study of what is known in the literature as integral-geometric transformation formulas of Blaschke--Petkantschin type. Such formulas go back to the pioneering works of W.\ Blaschke \cite{BlaschkeIntegralgeometrie} in dimensions $n=2$, and $n=3$ and have more systematically been investigated by his student B.\ Petkantschin \cite{Petkantschin}. They have further been developed in \cite{SantaloBook} using the language of differential forms and in \cite{SchWeil2008} by means of a measure-theoretic approach. Blaschke--Petkantschin formulas are fundamental devices in integral and stochastic geometry and have found various applications in convex geometry and geometric analysis \cite{ChasapisEtAl,DalMasoEtAl,DannEtAl,HaddadLudwig,LudwigFracPerim,MilmanYehudayoff,RubinDrury} as well as in stereology \cite{DareKiderlen2024,GardnerBook,GardnerJensenVolcicSurvey,JensenBook,KiderlenSurvey,Koetzer}. For the intersection of two linear subspaces a formula of this type is given by
\begin{equation}\label{eq:BPLinear}
	\begin{split}
		&\int_{G(n,q)}\int_{G(n,n-q+\gamma)}f(L_1,L_2)\,\nu_{n-q+\gamma}(dL_2)\nu_{q}(dL_1) \\
		&= c_1(n,q,\gamma)\int_{G(n,\gamma)}\int_{G(L,q)}\int_{G(L,n-q+\gamma)}f(L_1,L_2)[L_1,L_2]^{\gamma}\,\nu_{n-q+\gamma}^L(dL_2)\nu_{q}^L(dL_1)\nu_\gamma(dL),
	\end{split}
\end{equation}
where $f:G(n,q)\times G(n,n-q+\gamma)\to[0,\infty)$ is a measurable function and $G(L,q)$ is the relative Grassmannian of all $q$-dimensional linear subspaces of $\R^n$ containing $L$, whereas $\nu_{q}^L$ is the invariant probability measure on $G(L,q)$ which is invariant under all rotations of $\R^n$ that fix $L$ (similarly for $G(L,n-q+\gamma)$ and $\nu_{n-q+\gamma}^L$). We refer to \cite[Thm.~7.2.5]{SchWeil2008} where also the value of the constant $c_1(n,q,\gamma)$ can be found, which only depends on the parameters in brackets. The corresponding formula for the intersection of two affine subspaces is a special case of \cite[Thm.~7.2.8]{SchWeil2008} and reads as follows:
\begin{equation}\label{eq:BPAffine}
	\begin{split}
		&\int_{A(n,q)}\int_{A(n,n-q+\gamma)}f(E_1,E_2)\,\mu_{n-q+\gamma}(dE_2)\mu_{q}(dE_1) \\
		&= c_2(n,q,\gamma)\int_{A(n,\gamma)}\int_{A(E,q)}\int_{A(E,n-q+\gamma)}f(E_1,E_2)[E_1,E_2]^{\gamma+1}\,\mu_{n-q+\gamma}^E(dE_2)\mu_{q}^E(dE_1)\mu_\gamma(dE)    
	\end{split}    
\end{equation}
for measurable functions $f:A(n,q)\times A(n,n-q+\gamma)\to[0,\infty)$. Here, $A(E,k)$ stands for the family of $k$-dimensional affine subspaces containing $E$ and $\mu_{k}^E$ denotes the invariant measure on $A(E,k)$, $k=q$ and $k=n-q+\gamma$. The value of the constant $c_2(n,q,\gamma)$ only depends on the parameters in brackets and can be found in \cite{SchWeil2008}.

The present paper deals with a situation which in a sense is intermediate between \eqref{eq:BPLinear} and \eqref{eq:BPAffine}, and combines the linear with the affine set-up. To the best of our knowledge, this has not found attention so far in the literature. More explicitly, let $L\in G(n,q)$ be a $q$-dimensional linear subspace of $\R^n$, $q\in\{1,\ldots,n-1\}$, and for $\gamma\in\{0,\ldots,q-1\}$ let $E\in G(n,n-q+\gamma)$ be an affine subspace of dimension $n-q+\gamma$. If $E$ and $L$ are in general position, their intersection $E\cap L$ is an affine subspace of dimension $\gamma$. 

Our principal goal is the following. Find, for a given rotation invariant measure $\tilde \mu_{n-q+\gamma}$ on $A(n,n-q+\gamma)$ that is absolutely continuous with respect to $\mu_{n-q+\gamma}$, a measurable function $J: A(n,\gamma) \rightarrow [0,\infty)$ such that 
\begin{equation}\label{eq:Intro1}
	\int_{G(n,q)}   \int_{A(n,n-q+\gamma)} f(E \cap L) 
	\, \tilde\mu_{n-q+\gamma}(dE)\nu_q(dL)
	=
	\int_{A(n,\gamma)} f(E)\,J(E) \, \mu_\gamma (dE)
\end{equation}
holds for all measurable functions $f:A(n,\gamma)\to[0,\infty)$. The main result of this paper provides an explicit description of $J(E)$ and its dependence on $\tilde \mu_{n-q+\gamma}$. In the important particular case $\tilde \mu_{n-q+\gamma}=\mu_{n-q+\gamma}$, it turns out that $J(E)=c d(o,E)^{-(n-q)}$, where $c$ is a known constant and $d(o,E)$ stands for the distance of $E$ to the origin $o$. Another interesting case arises when
$\tilde \mu_{n-q+\gamma}$ is the restriction of $\mu_{n-q+\gamma}$ to $[hB^n]_{n-q+\gamma}$ for  some fixed $h>0$. Then, the left-hand side of
\eqref{eq:Intro1} is
\begin{equation}\label{eq:Intro2}
	\int_{G(n,q)}   \int_{A(n,n-q+\gamma)} f(E \cap L) \,1_{\{d(o,E)\leq h\}}
	\ \mu_{n-q+\gamma}(dE) \nu_q (dL),
\end{equation}
and $J(E)$  involves, besides of $d(o,E)^{-(n-q)}$, an additional factor that can be expressed in terms of an incomplete beta function. In probabilistic terms, our new integral-geometric transformation formula \eqref{eq:Intro1} will allow us to determine the density with respect to the invariant measure $\mu_\gamma$ on $A(n,\gamma)$ of the intersection of a random linear subspace $L$ of dimension $q$ and a stochastically independent random subspace $E$ of dimension $n-q+\gamma$ hitting the unit ball in $\R^n$, see Figure \ref{3Situations}. Moreover, we will also be able to determine the corresponding density with respect to $\mu_\gamma$ if the random affine subspace $E$ is only tangent to the unit sphere, see again Figure \ref{3Situations}.

\medskip

The remaining parts of this paper are structured as follows. In Section \ref{sec:Notation} we set up the notation and gather some background material. Some preliminary considerations are contained in Section \ref{sec:Preliminaries} and in Section \ref{sec:Results} we formulate our main theorems, especially the new integral-geometric transformation formula of Blaschke--Petkantschin-type. In Sections \ref{sec:Application1} and \ref{sec:Application2} we present the two applications to intersection probabilities mentioned above. Finally, Section \ref{sec:Proof} contains the proofs of our main results.

\section{Notation and background material}\label{sec:Notation}

Let $\mathbb{R}^n$ denote the  $n$-dimensional Euclidean space  for some fixed dimension $n\geq 1$. The \emph{Euclidean norm} will always be denoted by $\|\,\cdot\,\|$ and by $\lambda_n$ we indicate the \emph{Lebesgue measure} on $\mathbb{R}^n$. The \emph{Euclidean unit ball} and \emph{sphere} are denoted by $B^n$ and $S^{n-1}$ and their volume and surface content are given by
\begin{align}\label{eq:KappanOmegan}
 \kappa_n=\lambda_n(B^n) = \frac{\pi^{n/2}}{\Gamma(\frac{n}{2}+1)}\qquad\text{and}\qquad\omega_n=\mathcal{H}^{n-1}(S^{n-1}) = \frac{2\pi^{n/2}}{\Gamma(\frac{n}{2})},
\end{align}
respectively. Here, $\mathcal{H}^{n-1}$ is the \emph{$(n-1)$-dimensional Hausdorff measure} in $\R^n$. Further, for a set $E\subset \mathbb{R}^n$ we define the distance  $d(o,E)=\inf_{x\in E}\|x\|$ of $E$ to the origin $o\in\R^n$. We will make use of the \emph{incomplete beta function}
\[
B(x;\alpha,\beta)=\int_0^x t^{\alpha-1}(1-t)^{\beta-1}\,dt, \qquad 0\le x\le 1,
\]
with real parameters $\alpha,\beta>0$. Notice that the \emph{complete beta integral} 
satisfies 
\begin{equation}
B\big(\tfrac m2,\tfrac k2\big)=B\big(1;\tfrac m2,\tfrac k2\big)=2\frac{\omega_{m+k}}{\omega_m\omega_k},\qquad m,k\in \N. 
\label{eq:betaInt}
\end{equation}

For $n\geq 1$ and $q\in\{0,\ldots,n\}$ we let  $G(n,q)$ be the \emph{Grassmannian} of all $q$-dimensional linear subspaces of $\mathbb{R}^n$ and write  $A(n,q)$ for the family of $q$-dimensional affine subspaces of $\R^n$. We endow these spaces with the usual Borel $\sigma$-algebras and describe now shortly invariant measures on these families. Details can be found  e.g.~in \cite{SchWeil2008}, where also measurablility issues are discussed.

The group $SO_n$ of rotations in $\mathbb{R}^n$ carries a unique invariant (or \emph{Haar}) probability measure $\nu$. This group 
acts naturally on the space $G(n,q)$ and we denote by $\nu_q$ the unique $SO_n$-invariant probability measure on $G(n,q)$. Both, $SO_n$ and the group of translations act naturally on the \emph{affine Grassmannian} $A(n,q)$. There exists a motion invariant measure on  $A(n,q)$, and this measure is unique up to a multiplicative constant. We will use the motion invariant measure $\mu_q$ on the affine Grassmannian $A(n,q)$ given by 
\begin{align}\label{eq:DefMuq}
	\mu_q(\,\cdot\,) = \int_{G(n,q)}\int_{L^\perp}1_{\{L+x\in\,\cdot\,\}}\ \lambda_{L^\perp}(dx)\nu_q(dL),
\end{align}
where $\lambda_{L^\perp}$ denotes the Lebesgue measure on the orthogonal complement $L^\perp$ of $L\in G(n,q)$. Finally, for $M\in A(n,q)$ and $p\in\{0,\ldots,q\}$ we denote by $G(M,p)$ and $A(M,p)$ the \emph{relative Grassmannian} of all $p$-dimensional linear and affine subspaces contained in $M$, respectively. If, on the other hand, $p\in\{q,\ldots,n\}$ then $G(M,p)$ and $A(M,p)$ are the sets of linear and affine subspaces of dimension $p$ that contain $M$. These spaces carry natural invariant measures $\nu_p^M$ and $\mu_p^M$ as described in \cite[Sec.\ 7.1]{SchWeil2008}. In particular, these measures satisfy
\begin{align*}
	\int_{G(n,q)}\int_{G(M,p)}f(L)\ \nu_p^M(dL)\ \nu_q(dM) = \int_{G(n,p)}f(L)\ \nu_p(dL)
\end{align*}
for all measurable functions $f:G(n,p)\to[0,\infty)$, and similarly
\begin{align*}
	\int_{A(n,q)}\int_{A(F,p)}f(E)\ \mu_p^F(dE)\ \mu_q(dF) = \int_{A(n,p)}f(E)\ \mu_p(dE)
\end{align*}
for all measurable functions $f:A(n,p)\to[0,\infty)$, see \cite[Thm.~7.1.1 and Thm.~7.1.2]{SchWeil2008}.

Let $0\leq p,q\leq n-1$ and fix $L\in G(n,p)$ and $M\in G(n,q)$. If $p+q\leq n$ the \emph{subspace determinant} $[L,M]$ is defined as the $(p+q)$-volume of a parallelepiped spanned by the union of an orthonormal basis in $L$ and an orthonormal basis in $M$. If $p+q\geq n$ we define $[L,M]=[L^\perp,M^\perp]$. If $p+q=n$, both definitions coincide and $[L,M]$ is the factor by which the $p$-volume is multiplied under the orthogonal projection from $L$ onto $M^\perp$ thus,
\begin{equation}\label{eq:IntroProjectionFormula}
	\int_{L^\perp} f(x) \ \lambda_{L^\perp}(dx) 
	=
	[M,L]  \int_M f( x|L^\perp  ) \ \lambda_M (dx),
\end{equation}
for measurable $f:L^\perp\to [0,\infty)$. Here, 
$x | L^\perp$ denotes the orthogonal projection of a point $x\in \R^n$ onto $L^\perp$. For further background on subspace determinants, we refer to \cite[Sec.~14.1]{SchWeil2008}.

\section{Preliminary considerations}\label{sec:Preliminaries}

Let $n\geq 1$  and $q\in\{1,\ldots,n-1\}$.  The rotation group $SO_n$ acts on the space of real-valued functions $f$ on $A(n,q)$ by 
\[
(\eta f)(E)=f(\eta^{-1}E),\qquad E\in A(n,q), 
\]
for $\eta\in SO_n$. The function $f$ is called \emph{rotation invariant} if $\eta f= f$ for any rotation $\eta \in SO_n$. The \emph{rotational mean} of $f$, given by
\begin{equation*}
		f_{\mbox{\small rot}}(E)
		=
		\int_{SO_n} (\eta f)(E) \ \nu(d\eta),\qquad E\in A(n, q),
\end{equation*}
 is rotation invariant. 

Assume that  $f$ is rotation invariant and $n\ge 2$. Then, we have $f(E)=f(E')$ for any two affine subspaces $E,E'\in A(n,q)$ with $d(o,E)=d(o,E')$, since there is a rotation $\eta\in SO_n$ with $\eta E=E'$.  
This implies that there is a function  $f_I:[0,\infty)\to\R$, such that 
\begin{equation}\label{eq:fI}
	f(E) = f_I( d(o,E)), \qquad E \in A(n,q).
\end{equation}

The following lemma shows that for our purposes, it is essentially enough to consider the class of rotation invariant functions.

\begin{lemma}\label{Lemma:EnoughRotational}
	Let $n\geq 2$, $q\in\{1,\ldots,n-1\}$,  $\gamma\in\{0,\ldots,q-1\}$, a rotation invariant measure $\tilde \mu_{n-q+\gamma}$ on $A(n,n-q+\gamma)$ and a rotation invariant function $J:A(n, \gamma) \rightarrow [0, \infty)$ be given. Then 
	\begin{equation}\label{eq:lemma21}
		\int_{G(n,q)}   \int_{A(n,n-q+\gamma)} f(E \cap L) 
		\ \tilde \mu_{n-q+\gamma}(dE) \ \nu (dL)
		=
		\int_{A(n,\gamma)} f(E) \ J(E) \ \mu_\gamma (dE)
	\end{equation}
    holds for all measurable functions $f : A(n , \gamma) \rightarrow [0,\infty)$, if it holds for those $f$ that are in addition rotation invariant.  
\end{lemma}
\begin{proof}
	Suppose that \eqref{eq:lemma21} holds true for all non-negative, measurable and rotation invariant functions. Fix an arbitrary measurable function $f:A(n,q) \rightarrow [0,\infty)$. Using the rotation invariance of the measures $\tilde\mu_{n-q+\gamma}$ and $\nu_q$ and Tonelli's theorem, the left-hand side of \eqref{eq:lemma21} is 
   \begin{align*}
		&\int_{SO_n}
	\int_{G(n,q)}   \int_{A(n,n-q+\gamma)} f(\eta^{-1}(E \cap L)) 
		\ \tilde\mu_{n-q+\gamma}(dE) \ \nu_q (dL)
      \nu(d\eta)
		\\&=
		\int_{G(n,q)}   \int_{A(n,n-q+\gamma)}  f_{\mbox{\small rot}}( E)
		\ \tilde\mu_{n-q+\gamma}(dE) \ \nu_q (dL). 
	\end{align*} 
   Using the rotation invariance of 
   $\mu_\gamma$ and $J$, a similar argument shows that the right-hand side of \eqref{eq:lemma21} is 
	\begin{align*}
		\int_{A(n,\gamma)} 
          f_{\mbox{\small rot}}(E) 
          J(E) \, \mu_\gamma (dE). 
	\end{align*}
	By assumption,  \eqref{eq:lemma21} holds for the rotation invariant function $f_{\mbox{\small rot}}$, so the last two displayed expressions coincide and the assertion is shown. 
\end{proof}

In the proof of our main result, the following lemma will turn out to be of crucial importance. It can be seen a generalization of  \cite[Lem.~4.4]{HugSchneiderSchuster}. In that result the authors prove that 
\begin{equation}\label{eq:HugEtAl}
	A(n,k,r,\alpha) =
	\int_{G(n,k)}[F,L]^\alpha\,\nu_k(dL)= 
	\prod_{i=0}^{n-r-1}\frac{\Gamma(\frac{n-i}{2})\Gamma(\frac{k-i+\alpha}{2})}
	{\Gamma(\frac{n-i+\alpha}2)\Gamma(\frac{k-i}2)}, 
\end{equation}
for $\alpha\ge 0$, $r,k\in \{1,\ldots,n\}$ with $r+k\ge n$ and $F\in G(n,r)$, where the right-hand side is interpreted as $1$ if $r=n$. The next lemma is a counterpart for the Grassmannian associated to a hyperplane. (We note that the definition of the subspace determinant in \cite{HugSchneiderSchuster} is equivalent to our definition as $r+k\ge n$.)

\begin{lemma}\label{Lemma:IntegralSubspaceDet}
	Let $n\geq 2$, $\alpha\ge 0$ and $p,q\in \{1,\ldots,n-1\}$ with $p+q\le n$ be given. Then, for any $u \in S^{n-1}$ and any fixed $M\in G(n,p)$ we have that
	\begin{equation}\label{eq:roatatFixedAxisSubspDet}
		\int_{G(\myspan u,q)} [L,M]^\alpha\, \nu_q^{\myspan{u}}(dL)= a({n,p,q,\alpha})[u,M]^\alpha
	\end{equation}
	with 
	\[
	a({n,p,q,\alpha})=\prod_{i=1}^{p}
	\frac{\Gamma(\frac{n-i}{2})\Gamma(\frac{n-q-i+\alpha+1}{2})}
	{\Gamma(\frac{n-i+\alpha}2)\Gamma(\frac{n-q-i+1}2)}.
	\]
\end{lemma}
\begin{proof} If $u\in M$, then \eqref{eq:roatatFixedAxisSubspDet} holds trivially as both sides vanish. Hence, we may assume $u\not\in M$. Suppose first that we also have  $u\not\in M^\perp$ and fix $L\in G(\myspan{u},q)$. Since $\dim M^\perp+\dim L^\perp=2n-(p+q)\ge n$, \cite[Lem.~4.1]{Rataj1999} implies that 
	\begin{align*}
		[L^\perp,M^\perp]&=[L^\perp,M^\perp\cap u^\perp]\|u|M^\perp\|
		\\&=[L^\perp,M^\perp\cap u^\perp][u,M], 
	\end{align*}
	where we recall that $u |M^\perp$ is the projection of $u$ onto $M^\perp$.
	Here, we used that fact that the subspace determinant $[L',M']$, as defined in 
	\cite{Rataj1999}, coincides with our definition whenever $\dim M'+\dim L'\ge n$, but differs otherwise, causing 
	\cite{Rataj1999} to consider subspace determinants relative to $u^\perp$, which is not necessary using our definition. We thus obtain 
	\[
	[L,M]=[L^\perp,M^\perp]=[L^\perp,M^\perp\cap u^\perp][u,M], 
	\]
	a relation that is also true for  $u\in M^\perp$, implying that the left-hand side of \eqref{eq:roatatFixedAxisSubspDet} coincides with 
	\begin{align*}
		&\int_{G(\myspan{u},q)} [L^\perp,M^\perp\cap u^\perp]^a\, \nu_q^{\myspan{u}}(dL)\, [u,M]^\alpha
		=\int_{G(u^\perp,n-q)} [L,M^\perp\cap u^\perp]^a\, \nu_{q-1}^{u^\perp}(dL)\, [u,M]^\alpha. 
	\end{align*}
	The last integral is now of the form \eqref{eq:HugEtAl}, but with $u^\perp$ instead of $\R^n$ as the ambient space. It is thus equal to $A(n-1,n-q,n-p-1,\alpha)$. This constant coincides with $a({n,p,q,\alpha})$, and the assertion is proven. 
\end{proof}

\section{Presentation of the main results}\label{sec:Results}

Having established the basic notions and concepts in Section \ref{sec:Preliminaries}, we will now state our main result: a general reduction of integrals of the form \eqref{eq:Intro1}. 
We remark that we imposed the dimensional constraints $0\le \gamma<q<n$ at the beginning of the introduction, since $\gamma=q$ or $q=n$,  would imply that the left-hand side of  \eqref{eq:Intro1} becomes trivial. These constraints imply the assumption $n\ge 2$,  which we now adopt for the rest of the paper.

Recall that \eqref{eq:Intro1} involves a rotation invariant measure 
$\tilde \mu_{n-q+\gamma}$ on $A(n,n-q+\gamma)$ that is dominated by $\mu_{n-q+\gamma}$. Hence, the Radon–Nikodym theorem guarantees the existence of a $\mu_{n-q+\gamma}$-density $\tilde H\ge 0$ for $\tilde \mu_{n-q+\gamma}$. The assumed rotation invariance implies that 
$H=\tilde H_{\mbox{\small rot}}$ is also a $\mu_{n-q+\gamma}$-density for $\tilde \mu_{n-q+\gamma}$. We will state our results using this density, as the function $J(r)=J_H(r)$ can explicitly be expressed in terms of $H$. 
For specific choices of the density $H$, the function $J_H(r)$ in the statement of this theorem can be simplified and be made more explicit, as illustrated in Corollaries  \ref{cor:Hindicator} and \ref{cor:MultipleIntersections} below. The proof of Theorem \ref{ThmGeneral} is postponed to Section \ref{sec:Proof} at the end of this paper. 
%Our main result reads now as follows.

\begin{theorem}\label{ThmGeneral}
Fix $n\geq 2$, $q\in\{1,\ldots,n-1\}$, $\gamma\in\{0,\ldots,q-1\}$,  and let $H: A(n,n-q+\gamma) \rightarrow [0,\infty)$ be a measurable and rotation invariant function. Then 
\begin{align*}
    &\int_{G(n,q)}
		\int_{A(n,n-q+\gamma)}
		f(E \cap L) H(E) \ \mu_{n-q+\gamma}(dE) \ \nu_q(dL)
  \\&=D(n,q,\gamma)
		\int_{A(n, \gamma)}
		f( E) d(o,E)^{-(n-q)} J_H(d(o,E)) \ \mu_\gamma(dE)
\end{align*}
for all measurable $f: A(n,\gamma) \rightarrow [0,\infty)$. Here, 
	\begin{equation}\label{eq:defJ}
		J_H(r) =\int_0^1
		H_I(r z) z^{ q}(1-z^2)^{ \frac{n-q}{2}-1}
		\ dz,
	\end{equation}
and the constant is given by
	\begin{equation}
		D(n,q,\gamma)=  \frac{\omega_{\gamma+1}\omega_{q-\gamma}
			\omega_{n-q}
		}{\omega_{n-(q-\gamma)+1}\omega_{n-\gamma}}.
  \label{eq:Ddef}
	\end{equation}
\end{theorem}

%Here we will just briefly describe the proof strategy such that it can be placed in relation to the previous lemmas. As $f$ can be assumed to be rotation invariant by Lemma \ref{Lemma:EnoughRotational}, \eqref{eq:fI} is available and the proof can be divided into three steps. \fxnote{MK: I am not sure if the three items below are helpful. I wouldn't understand them if I didn't know the proof. In 2: is "the distance" really "expressed as multiple integral"? I would omit the description 1.-3.}
%\begin{itemize}
%	\item[1.] The first step consists in understanding $d(o,E \cap L)$ for $E \in A(n,n-q+\gamma)$ and $L \in G(n,q)$. As the dimension of the generic intersection is $\gamma$ we can use the integral-geometric transformation formula \eqref{eq:BPAffine} to reduce the dimension. This leads to the problem of understanding $d(o, \Tilde{E}\cap \Tilde{L})$ for $\Tilde{E}\in A(n-\gamma, n-q)$ and $\tilde{L}\in G(n- \gamma, q-\gamma)$. In essence, the purpose of Step 1 is to reduce the general problem to a setting where the intersection is a point.
	
%	\item[2.] Step 2 uses \eqref{eq:IntroProjectionFormula} to describe the distance of the origin from the intersection point arising from the reduction in Step 1 as a rather involved multiple  integral.
	
%	\item[3.] In the final step the multiple integral of Step 2 is simplified by means of Lemma \ref{Lemma:IntegralSubspaceDet} and \cite[Lemma 1]{Auneau2010}. Eventually, this yields Theorem \ref{ThmGeneral} after the value of the constant $D(n,q,\gamma)$ has been identified.
%\end{itemize}

We shall now discuss two special cases in which the function $J_H(r)$ and thus the integral relation  in Theorem \ref{ThmGeneral} can be simplified. We start by considering the special case, where $H$ is the  function $H_h(E) = 1_{\{d(o,E) \leq h\}}$ for some $h>0$, 
corresponding to $\tilde \mu_{n-q+\gamma}$ in \eqref{eq:Intro1} being the restriction of $\mu_{n-q+\gamma}$ to $[hB^n]_{n-q+\gamma}$. Definition \eqref{eq:defJ} and a substitution yield 
\begin{align}
 \label{eq:JHr}
	J_{H_h}(r)
	&= \frac12
	\int_0^{(\min\{\frac{h}{r},1\})^2}
	z^{\frac{q+1}2-1}(1-z)^{ \frac{n-q}{2}-1}
	\ dz
	=
	\begin{cases}
\frac{\omega_{n+1}}{\omega_{q+1}\omega_{n-q}}  &: r \leq h,
		\\
		\frac1 2B\big((\tfrac{h}{r})^2;\frac{q+1}2,\frac{n-q}2\big)  &: r > h,
	\end{cases}
\end{align}
where \eqref{eq:betaInt} was used at the second equality sign.  This implies the following result.

\begin{corollary}\label{cor:Hindicator}
	Fix $n\geq 2$, $q\in\{1,\ldots,n-1\}$, $\gamma\in\{0,\ldots,q-1\}$, and $h>0$. Then
	\begin{align*}
		&\int_{G(n,q)}
		\int_{[hB^n]_{n-q+\gamma}}
		f(E \cap L)\, \mu_{n-q+\gamma}(dE) \, \nu_q(dL)
		\\&=D(n,q,\gamma)
		\int_{A(n,\gamma)} f(E) d(o,E)^{-(n-q)} J(d(o,E)) 
  \, \mu_\gamma(dE),
	\end{align*}
	where $J=J_{H_h}$ is given by \eqref{eq:JHr} and  where $D(n,q,\gamma)$ is the constant given in Theorem \ref{ThmGeneral}.
\end{corollary}
The incomplete beta integral, and hence the function $J$ in Corollary \ref{cor:Hindicator}, can be expressed in terms of a hypergeometric function. Also, weight functions $J_H$ associated to more general densities $H$ can be expressed in terms of -- possibly several -- hypergeometric functions. For instance, if $H_I$ is an even polynomial restricted  to $[0,\infty)$, Euler's integral relation implies such a representation,  see e.g.~\cite{AbramowitzStegun} for details. 

Using the monotone convergence theorem when $h\to\infty$ in Corollary 
\ref{cor:Hindicator}, gives an explicit 
integral relation of the form  \eqref{eq:Intro1} with $\tilde 
\mu_{n-q+\gamma}=\mu_{n-q+\gamma}$.

\begin{corollary}\label{thm:Hconstant}
	Fix $n\geq 2$, $q\in\{1,\ldots,n-1\}$ and $\gamma\in\{0,\ldots,q-1\}$. Then  
	\begin{align*}
		&\int_{G(n,q)}
		\int_{A(n,n-q+\gamma)}
		f(E \cap L)
		\,\mu_{n-q+\gamma} (dE)\,\nu_q(dL)
		\\&=
  \tilde D(n,q,\gamma) \int_{A(n,\gamma)} f(E) d(o,E)^{-(n-q)}  \ \mu_\gamma(dE)
	\end{align*}
 for all measurable functions $f:A(n, \gamma) \rightarrow [0,\infty)$. Here, 
	\begin{equation*} 
 \tilde D(n,q,\gamma)=
\frac{\omega_{n+1}\omega_{\gamma+1}
     \omega_{q-\gamma}}
     {\omega_{n-(q-\gamma)+1} 
     \omega_{n-\gamma}\omega_{q+1}}. 
\end{equation*}
\end{corollary}

Let us also report that Theorem \ref{ThmGeneral} allows an extension to multiple intersections in the following way. The proof is postponed to Section \ref{sec:Proof}.

\begin{corollary}  \label{cor:MultipleIntersections}
Fix $\ell,m,n\geq 1$, $q_1,\ldots,q_\ell\in\{0,\ldots,n\}$, put $q=q_1+\dots+q_\ell-(\ell-1)n$. Suppose that $q\in\{1,\ldots,n-1\}$ and fix $\gamma\in\{0,\ldots,q-1\}$
 and $p_1,\ldots,p_m\in\{0,\ldots,n\}$
 such that $p_1+\dots+p_m-(m-1)n=n-q+\gamma$. Further, let $f:A(n,\gamma)\to[0,\infty)$ be a measurable function, and let $H:A(n,n-q+\gamma)\to[0,\infty)$ be
 a rotation invariant measurable function. Define
	\begin{align*}
		I_{\ell,m} &=\int_{G(n,q_1)}\cdots\int_{G(n,q_\ell)}\int_{{A(n,p_1)}}\cdots\int_{A(n,p_m)}f(E_1\cap\ldots\cap E_m\cap L_1\cap\ldots\cap L_\ell)\\
		&\hspace{3cm}\times H(E_1\cap\ldots\cap E_m)\,\mu_{p_m}(dE_m)\dots\mu_{p_1}(dE_1)\nu_{q_\ell}(dL_\ell)\dots\nu_{q_1}(dL_1).
	\end{align*}
	Then
	$$
	I_{\ell,m} = D(n,q,\gamma)
     {\omega_{n-q+\gamma+1}\omega_{n+1}^{m}\over\omega_{n+1}\omega_{p_1}\cdots\omega_{p_m}}\,	
		\int_{A(n, \gamma)}
		f( E) d(o,E)^{-(n-q)} J_H(d(o,E)) \ \mu_\gamma(dE),
	$$
	where $J_H$ is given by \eqref{eq:defJ} and the leading constant by \eqref{eq:Ddef}. 
\end{corollary}

Finally, we state an alternative version of Theorem \ref{ThmGeneral}, which appears to be more general, as the intersecting linear subspace now is fixed. Closer investigation shows, however, that the two statements are equivalent, if a suitable Blaschke--Petkantschin relation is applied. The proof of this equivalence will be given in Section \ref{sec:Proof}. 

\begin{theorem}\label{ThmGeneral1}
Fix $n\geq 2$,  $q\in\{1,\ldots,n-1\}$, $\gamma\in\{0,\ldots,q-1\}$  and 
$L_0\in G(n,q)$. Let $H: A(n,n-q+\gamma) \rightarrow [0,\infty)$ be a 
rotation invariant, measurable function. Then 
\begin{align}  \label{eq:new}
		\int_{A(n,n-q+\gamma)}
		f(E \cap L_0) H(E) \ \mu_{n-q+\gamma}(dE)=
  \frac
  {\omega_{\gamma+1}\omega_{n-q}}{\omega_{n-(q-\gamma)+1}}
		\int_{A(L_0, \gamma)}
		f(E) J_H(d(o,E)) \ \mu_\gamma^{L_0}(dE)
\end{align}
for all measurable $f: A(L_0,\gamma) \rightarrow [0,\infty)$. Here, $J_H$ is given in Theorem \ref{ThmGeneral}. 
\end{theorem} 

Similar versions, that is, versions where the linear space is hold fixed,  are possible also for Corollaries \ref{cor:Hindicator}-\ref{cor:MultipleIntersections}.

%In this we notice something of interest. Suppose that $f: A(n, \gamma) \rightarrow \mathbb{R}_+$ is measurable and satisfies $f(E)=0$ %whenever $d(o,E) >h_0$. Then, for any $h>h_0$,
%\begin{align*}
%    F(h)
%    &=
%    C_2
%        \int_{A(n,\gamma)} f(E) d(o,E)^{-(n-q)} J_h(d(o,E)) \ \mu_\gamma(dE)
%        \\
%        &=
%         C_2\int_{A(n,\gamma)} 1_{d(o,E) \leq h_0} f(E) d(o,E)^{-(n-q)} J_h(d(o,E)) \ \mu_\gamma(dE)
%         \\
%         &=C_2
%         \frac{\Gamma( \frac{q+1}{2}) \Gamma( \frac{n-q}{2})
	%}{2\Gamma( \frac{n+1}{2})}  \int_{A(n,\gamma)} 1_{d(o,E) \leq h_0} f(E) d(o,E)^{-(n-q)} \ \mu_\gamma(dE)
%\\
%&=
%C_3 
%  \int_{A(n,\gamma)} f(E) d(o,E)^{-(n-q)} \ \mu_\gamma(dE)
%\end{align*}
%with $C_3$ given in corollary \ref{thm:Hconstant}. By corollary \ref{cor:Hindicator} and \ref{thm:Hconstant} we conclude 
%\begin{equation*}
%     \int_{G(n,q)}   \int_{A(n,n-q+\gamma)} f(E \cap L) 1_{d(o,E)\leq h}
%    \ \mu(dE) \ \nu (dL)
%    =
%     \int_{G(n,q)}   \int_{A(n,n-q+\gamma)} f(E \cap L)
%    \ \mu(dE) \ \nu (dL)
%\end{equation*}
%for $h> h_0$. As the left hand side is a finite measure it is thus possible to estimate the right hand side by simulations of $E \cap L$.  
%\mkE{I looked at it again. The final result is no surprise. If $d(o,E) >h$ then $d(o,E \cap L) >h$ so $f(E \cap L)=0$ if $h > h_0$. Thus %what is show is almost trivial - although it shows the behavior of $J_h(\cdot)$ behaves as wanted. }

\section{Intersection probabilities for linear and affine subspaces hitting the unit ball}\label{sec:Application1}

In this section we return to the  problem from stochastic geometry already mentioned in the introduction. Namely, we consider a random linear subspace $L$ of dimension $q\in\{1,\ldots,n-1\}$ in $\R^n$ with distribution $\nu_q$ and a random affine subspace $E$ of dimension $n-q+\gamma$, where $\gamma\in\{0,\ldots,q-1\}$ is a fixed number. Recalling that
$$
[B^n]_{n-q+\gamma}=\{E\in A(n,n-q+\gamma):E\cap B^n\neq\emptyset\},
$$
we use the restriction of $\kappa_{q-\gamma}^{-1}\mu_{n-q+\gamma}$ to $[B^n]_{n-q+\gamma}$ as distribution for $E$ and suppose that $L$ and $E$ are stochastically independent. In fact, it follows from \eqref{eq:DefMuq} that $\kappa_{q-\gamma}^{-1}\mu_{n-q+\gamma}$ is indeed a probability measure on $[B^n]_{n-q+\gamma}$.
We are interested in the distribution of the random affine subspace $E\cap L$, which is almost surely of dimension $\gamma$. Since its distribution is invariant under rotations of $\R^n$, all relevant information is contained in the distribution of the random variable $d(o,E\cap L)$ describing the distance of $E\cap L$ to the origin.
This distribution turns out to have a density $f_{n,q,\gamma}$ with respect to Lebesgue measure on $[0,\infty)$. The next result shows that it  is heavy tailed and  asymptotically of Pareto type with shape parameter $\gamma+2$. 
\begin{theorem}\label{thm:Application}
	In the set-up just introduced, the probability density $f_{n,q,\gamma}(\delta)$ of the random variable $d(o,E\cap L)$ is given by
	\begin{align*}
		f_{n,q,\gamma}(\delta) = (q-\gamma){\omega_{\gamma+1}\omega_{n-q}\over\omega_{n-(q-\gamma)+1}}\begin{cases}
			{\omega_{n+1}\over\omega_{q+1}\omega_{n-q}}\delta^{q-\gamma-1} &: 0\le \delta\leq 1\\
			{1\over 2}\delta^{q-\gamma-1}B({1\over \delta^2};{q+1\over 2},{n-q\over 2}) &: \delta>1.
		\end{cases}
	\end{align*}
\end{theorem}
\begin{proof}
	Fix $\delta>0$ and consider the probability of the event $d(o,E\cap L)\leq\delta$. Using Corollary \ref{cor:Hindicator} with $h=1$, we obtain
	\begin{align*}
		\mathbb{P}[d(o,E\cap L)\leq\delta] &= {1\over\kappa_{q-\gamma}}\int_{G(n,q)}\int_{A(n,n-q+\gamma)}1_{\{d(o,E)\leq 1\}}\,1_{\{d(o,E\cap L)\leq\delta\}}\,\mu_{n-q+\gamma}(dE)\nu_q(dL)\\
		&={D(n,q,\gamma)\over\kappa_{q-\gamma}}\int_{A(n,\gamma)}1_{\{d(o,E)\leq\delta\}}d(o,E)^{-(n-q)}J_{H_1}(d(o,E))\,\mu_\gamma(dE),
	\end{align*}
 with
	$$
	J_{H_1}(r) = \begin{cases}
		{\omega_{n+1}\over\omega_{q+1}\omega_{n-q}} &: r\leq 1\\
		{1\over 2}B({1\over r^2};{q+1\over 2},{n-q\over 2}) &: r>1.
	\end{cases}
	$$
	We now apply the decomposition \eqref{eq:DefMuq}, use $d(o,L+x)=\|x\|$ whenever $x\in L^\perp$,  
    and then introduce spherical coordinates in 
    $L^\perp$ to see that 
	\begin{align*}
		&\mathbb{P}[d(o,E\cap L)\leq\delta]\\
		%&= {D(n,q,\gamma)\over\kappa_{q-\gamma} }\int_{G(n,\gamma)}\int_{L^\perp}1_{\{d(o,L+x)\leq\delta\}}d(o,L+x)^{-(n-q)}J_{H_1}(d(o,L+x))\,\lambda_{L^\perp}(dx)\nu_\gamma(dL)\\
		&={D(n,q,\gamma)\over\kappa_{q-\gamma}}
  \int_{G(n,\gamma)}\int_{L^\perp}1_{\{\|x\|\leq\delta\}}\|x\|^{-(n-q)}J_{H_1}(\|x\|)\,\lambda_{L^\perp}(dx)\nu_\gamma(dL)\\
		&={D(n,q,\gamma) \omega_{n-\gamma}\over\kappa_{q-\gamma}}\int_0^\delta r^{q-\gamma-1}J_{H_1}(r)\,dr.
	\end{align*}
	Taking the derivative with respect to $\delta$ in the last expression, inserting the value of the constant $D(n,q,\gamma)$ of Theorem \ref{ThmGeneral} and then using \eqref{eq:KappanOmegan} yields the result.
\end{proof}

Although the density $f_{n,q,\gamma}(\delta)$ is defined piecewise, we remark that it is continuous at the point $\delta=1$. In fact, this follows from the continuity of the incomplete beta function in combination with \eqref{eq:betaInt}. Let us also mention at this point that since we are working with the unit ball $B^n$ as a reference set, the results of this section and also the next one continue to hold if the random linear subspace $L$ is replaced by a deterministic linear subspace of the same dimension.

Since the density of $d(o,E\cap L)$ is asymptotically of Pareto type with shape parameter $\gamma+2$, the moment properties of the random variable $d(o,E\cap L)$ depend on the intersection dimension $\gamma$. The next result delivers a precise description. 

\begin{corollary}
	Let $\alpha\in\R$. In the set-up just introduced, one has that $\mathbb{E}d(o,E\cap L)^\alpha<\infty$ if and only if $\alpha\in(\gamma-q,\gamma+1)$. In particular, if $\gamma=0$ the random variable $d(o,E\cap L)$ has infinite expectation.
\end{corollary}
\begin{proof}
	We have to check under which conditions on $\alpha$ the product of $\delta^\alpha$ with the probability density $f_{n,q,\gamma}(\delta)$ of Theorem \ref{thm:Application} is integrable at $\delta=0$ and $\delta=\infty$. The function $\delta\mapsto\delta^{q-\gamma-1+\alpha}$ is integrable at $\delta=0$, if and only if $\alpha>\gamma-q$. Moreover,  since $B(\delta^{-2};a,b)={\delta^{-2a}\over a}+O(\delta^{-2(a+1)})$ as $\delta\to\infty$, see \cite[\S 6.6.8 and \S 15.7]{AbramowitzStegun}, the required integrability is satisfied, whenever the function $\delta^{\alpha+q-\gamma-1-(q+1)}=\delta^{\alpha-\gamma-2}$ is integrable at $\delta=\infty$. The latter holds if and only if $\alpha-\gamma-2<-1$, or equivalently, $\alpha<\gamma+1$. This completes the proof.
\end{proof}

As anticipated above, the intersection of $E$ and $L$ may  or may not hit the unit ball. By Theorem \ref{thm:Application} the probability for the first of these events is given by
\begin{align*}
	\mathbb{P}[E\cap L\cap B^n\neq \emptyset] &= (q-\gamma){\omega_{\gamma+1}\omega_{n-q}\over\omega_{n-(q-\gamma)+1}}{\omega_{n+1}\over\omega_{q+1}\omega_{n-q}}\int_0^1\delta^{q-\gamma-1}\,d\delta.
	%&={\omega_{\gamma+1}\omega_{n+1}\over\omega_{q+1}\omega_{n-(q-\gamma)+1}}.
\end{align*}
Simplification leads to the following result.
%Then, using \eqref{eq:KappanOmegan} we conclude the following result. \fxnote{ED : I am not sure if we use \eqref{eq:KappanOmegan} at this stage?, isn't it only after the proof of cor 5.3}

\begin{corollary}
	In the set-up just introduced, it holds that
	$$
	p_{n,q,\gamma}=\mathbb{P}[E\cap L\cap B^n \neq \emptyset] = {\omega_{\gamma+1}\omega_{n+1}\over\omega_{q+1}\omega_{n-(q-\gamma)+1}}.
	$$
\end{corollary}

This result allows quantifying the asymptotic behavior of the intersection probability $p_{n,q,\gamma}$  for fixed $q$ and $\gamma$  in high dimensions, that is, as $n\to\infty$. Using \eqref{eq:KappanOmegan}, we can rewrite $p_{n,q,\gamma}$ in terms of gamma functions as
\begin{align*}
	p_{n,q,\gamma} = {\Gamma({q+1\over 2})\Gamma({n-(q-\gamma)+1\over 2})\over\Gamma({\gamma+1\over 2})\Gamma({n+1\over 2})}.
\end{align*}
Since $q$ and $\gamma$ are assumed to be fixed, we can apply Stirling's formula \cite[\S 6.1.37]{AbramowitzStegun} for the gamma function to see that 
$$
p_{n,q,\gamma} = {\Gamma({q+1\over 2})\over \Gamma({\gamma+1\over 2})}\Big({2\over n}\Big)^{q-\gamma\over 2}(1+o_n(1)) = {\omega_{\gamma+1}\over\omega_{q+1}}\Big({2\over \pi\,n}\Big)^{q-\gamma\over 2}(1+o_n(1)),
$$
as $n\to\infty$, where we write $o_n(1)$ for a sequence that converges to zero with $n$. Hence,  in high dimensions, the intersection point of the random affine subspace $E$ and the random linear subspace $L$ will asymptotically almost surely be outside the unit ball $B^n$.

\begin{figure}[h!]
	\centering
	\includegraphics[width=0.3\columnwidth]{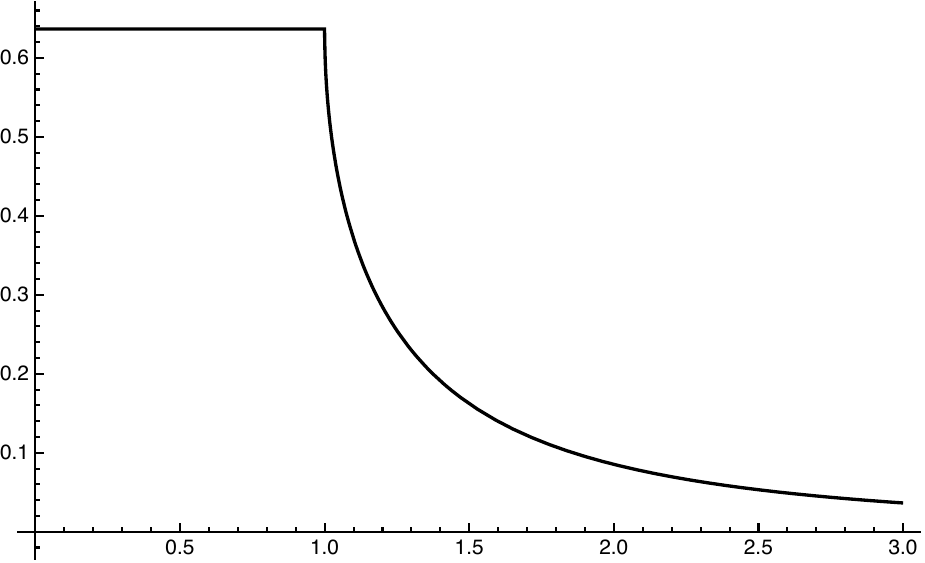}\quad
	\includegraphics[width=0.3\columnwidth]{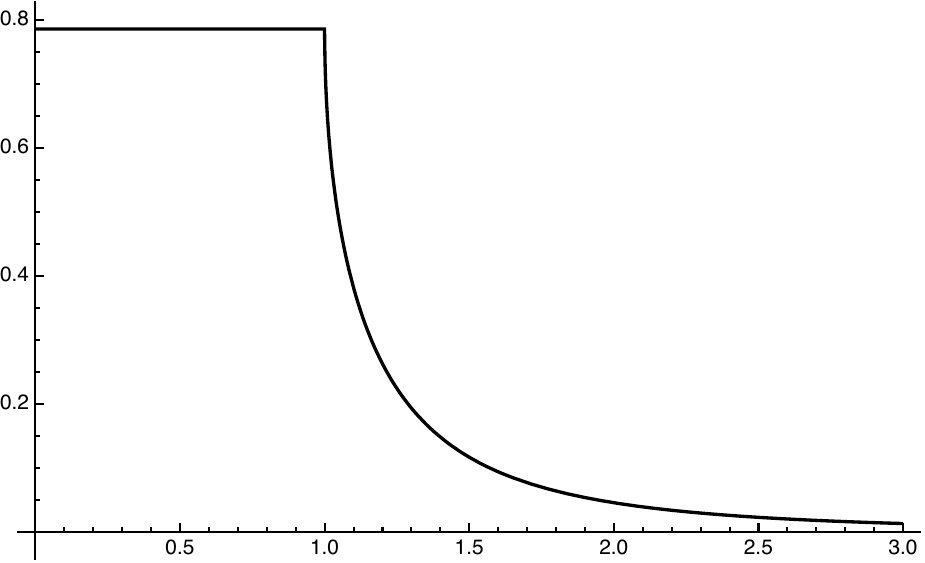}\quad
	\includegraphics[width=0.3\columnwidth]{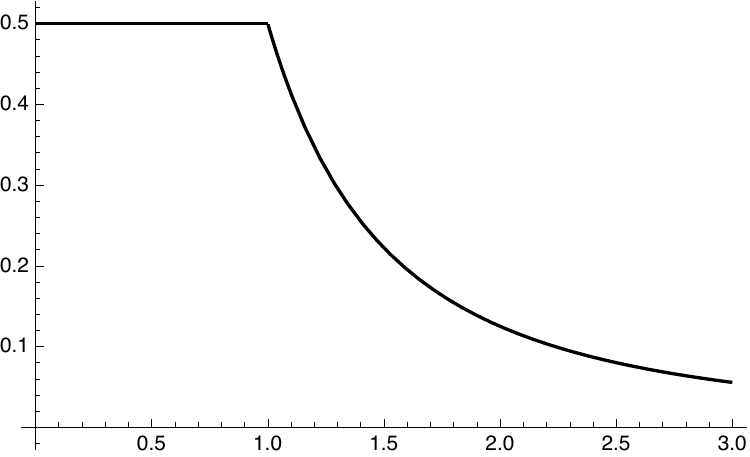}\\[0.5cm]
	\includegraphics[width=0.3\columnwidth]{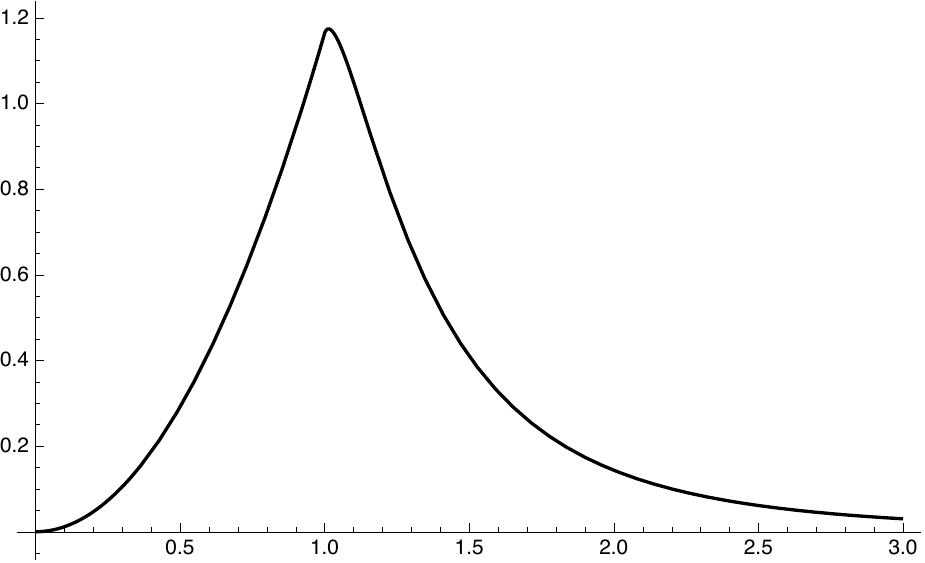}\quad
	\includegraphics[width=0.3\columnwidth]{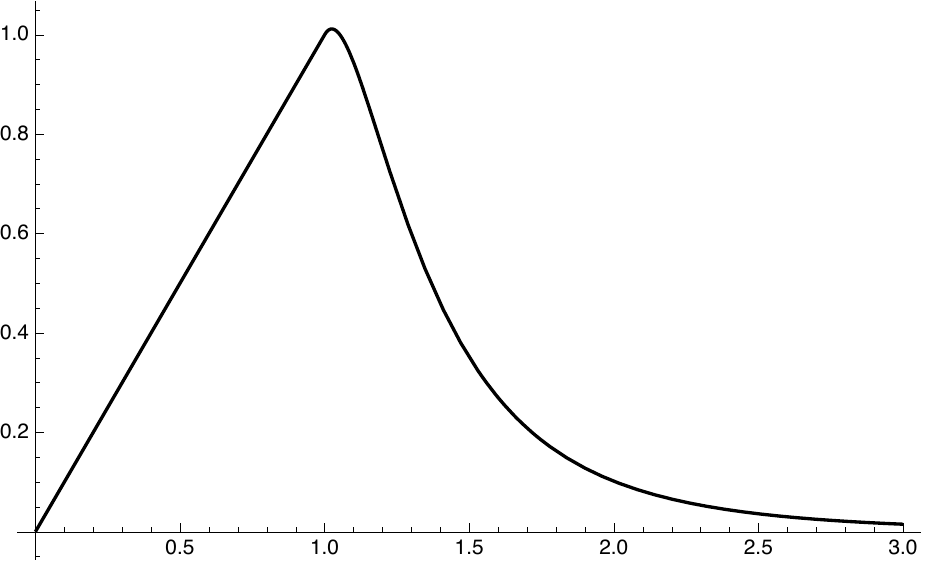}\quad
	\includegraphics[width=0.3\columnwidth]{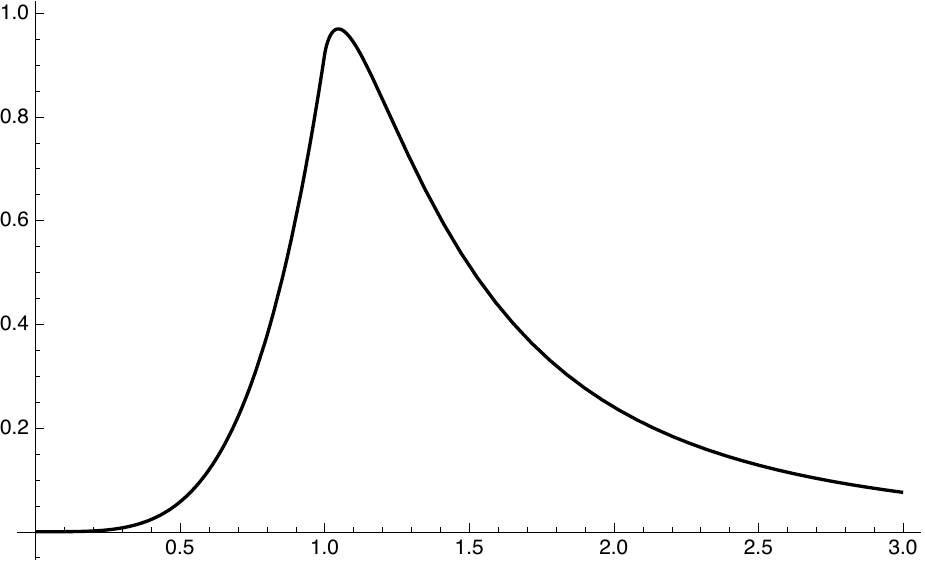}
	\caption{First line: The probability densities $f_{2,1,0}$, $f_{3,2,1}$ and $f_{3,1,0}$ from left to right. Second line: The probability densities $f_{8,5,2}$, $f_{9,5,3}$ and $f_{9,6,1}$ from left to right.}
	\label{fig:densities}
\end{figure}

In the planar case $n=2$ we can take $q=1$ and $\gamma=2$ and obtain
\begin{alignat*}{3}
	f_{2,1,0}(\delta) &= {2\over\pi}\,\begin{cases}
		1 &: 0\le\delta\leq 1\\
		1-\sqrt{1-\delta^{-2}} &: \delta>1,
	\end{cases}
	\qquad\qquad &&p_{2,1,0} &&= {2\over\pi}.
	\intertext{Moreover, if $n=3$ we can consider the case $q=2$ and $\gamma=1$, which yields}
	f_{3,2,1}(\delta) &= \begin{cases}
		{\pi\over 4} &: 0\le\delta\leq 1\\
		{1\over 2}(\operatorname{arccsc}(\delta)-\sqrt{1-\delta^{-2}}) &: \delta>1,
	\end{cases}
	\qquad\qquad &&p_{3,2,1}&&={\pi\over 4},
	\intertext{and the case $q=1$ and $\gamma=0$, where the probability density reduces to}
	f_{3,1,0}(\delta) &= {1\over 2}\begin{cases}
		1 &:0\le \delta\leq 1\\
		\delta^{-2} &: \delta> 1,
	\end{cases}
	\qquad\qquad &&p_{3,1,0}&&={1\over 2},
\end{alignat*}
see also Figure \ref{fig:densities}. In particular, if $n=3$, $q=2$ and $\gamma=1$, we obtain
$$
\mathbb{E} d(o,E\cap L)=\int_0^\infty \delta f_{3,2,1}(\delta)\,d\delta = {\pi\over 4},
$$
whereas in the other two cases $\mathbb{E} d(o,E\cap L)=\infty$. To illustrate the potential complexity of the density functions $f_{n,q,\gamma}(\delta)$ we also record the following values:

\setlength{\tabcolsep}{10pt} 
\renewcommand{\arraystretch}{1.5}
\begin{table}[H]
    \centering
    \begin{tabular}{c|l|c|c}
         $n,q,\gamma$ & $f_{n,q,\gamma}$ & $p_{n,q,\gamma}$ & $\mathbb{E}d(o,E\cap L)$\\
         \hline
         $8,5,2$ & $\begin{cases}
		{128\over 35\pi}\delta^2 &: 0\le\delta\leq 1\\
		{48\over 105\pi}{8\delta^7-(8\delta^6+4\delta^4+3\delta^2-15)\sqrt{\delta^2-1}\over\delta^5} &: \delta>1,
	\end{cases}$ & ${128\over 105\pi}$ & ${4\over\pi}$\\
        $9,5,3$ & $\begin{cases}
		\delta &: 0\le \delta\leq 1\\
		{4\delta^2-3\over\delta^7} &: \delta>1,
	\end{cases}$ & ${1\over 2}$ & ${16\over 15}$\\
        $9,6,1$ & $\begin{cases}
		{75\pi\over 265}\delta^4 &: 0\le\delta\leq 1\\
		{75\delta^8\operatorname{arccsc}(\delta)-5(16\delta^6+10\delta^4+8\delta^2-48)\sqrt{\delta^2-1}\over 128\delta^4} &: \delta>1.
	\end{cases}$ & ${15\pi\over 256}$ & ${5\pi\over 8}$
    \end{tabular}

\end{table}

We see that the complexity of the probability density $f_{n,q, \gamma}$ varies depending on $n,q,\gamma$. We notice in particular that $p_{9,5,3}=\frac{1}{2}$, which is the same as $p_{3,1,0}$ from before. In fact, for all odd $n$, there are (often multiple) ways we can choose $q$ and $\gamma$ such that $p_{n,q, \gamma}= \frac{1}{2}$. For instance, if we take $\gamma= \frac{n-1}{2}-1$ and $q = \frac{n-1}{2}+1$ or $q = n-2$, then $p_{n,q,\gamma}=\frac{1}{2}$. %For even $n$ it appears that $p_{n, q, \gamma}\ne \frac{1}{2}$ for all $ 0 \leq \gamma < q<n$.
%\fxnote{I don't see what "it appears" means. I'd rather delete this phrase!?}

\section{Intersection probabilities for linear and affine subspaces tangent to the unit sphere\label{sec:Application2}}

In this section, we consider another application of Theorem \ref{ThmGeneral} to intersection probabilities in stochastic geometry. As in the previous section, we consider a random linear subspace $L$ of dimension $q \in \{1, \ldots ,n-1\}$ in $\mathbb{R}^n$ with distribution $\nu_q$. In contrast, $E$ is now a stochastically independent random affine subspace $E$ in $\R^n$ of dimension $n-q+\gamma$ such that $d(o,E)=1$. Thus, 
%$E \in \{ E' \in A(n,n-q+\gamma) \ : \ d(o,E') =1 \}$, which means that
$E$ is tangent to the unit sphere $S^{n-1}$ in $\R^n$. The rotation invariant probability measure on that space is given by
\begin{equation}\label{eq:randomDensFlatSphere}
	\sigma_{n-q+\gamma}
	(A)
	=
	\frac{1}{\omega_{q-\gamma}}
	\int_{G(n,n-q+\gamma)}
	\int_{S^{n-1}\cap M^\perp} 1_A(M+u) \ \mathcal{H}^{q-\gamma-1}(du) \ \nu_{n-q+\gamma}(dM)
\end{equation}
for a Borel sets $A\subseteq\{ E' \in A(n,n-q+\gamma) \ : \ d(o,E') =1 \}$. As in the previous section, we are interested in the distribution of $L\cap E$. Again, all relevant information is contained in the distribution of $d(o,E \cap L)$ due to rotation invariance. The next result is a stepping stone in the derivation of the probability density of the random variable $d(o,E \cap L)$.

\begin{lemma}\label{cor:Sphere}
Fix $n\geq 2$, $q\in\{1,\ldots,n-1\}$, $\gamma\in\{0,\ldots,q-1\}$ and let $f: A(n,\gamma) \rightarrow [0,\infty)$ be a measurable, rotation invariant and bounded function.
Then
	\begin{align*}
		\int_{G(n,q)}   &\int_{G(n,n-q+\gamma)} \int_{S^{n-1} \cap M^\perp }
		f((M+hu) \cap L)
		\ \mathcal{H}^{q-\gamma-1}(du)
		\ \nu_{n-q+\gamma}(dM)  \ \nu_q (dL)
		\\&
		=
		D(n,q,\gamma)
	{\omega_{n-\gamma}}h^{\gamma+1}
		\int_h^\infty 
		f_I(r) r^{-(\gamma+2)}  \Big( 1-\frac{h^2}{r^2} \Big) ^{\frac{n-q}{2}-1}
		\ dr
	\end{align*}
 for almost every $h>0$. 
\end{lemma}
\begin{proof}
	Let $f:A(n, \gamma) \rightarrow [0, \infty)$ satisfy the above assumptions and define the function $F:(0,\infty)\to [0, \infty)$ by 
	\begin{align*}
		F(h)
		&=
		D(n,q,\gamma)\int_{A(n,\gamma)} f(E) d(o,E)^{-(n-q)} J_{H_h}(d(o,E)) \ \mu_\gamma(dE)
		\\& =D(n,q,\gamma)
	\omega_{n-\gamma} \int_0^\infty f_I(r) r^{q-\gamma-1}
		J_{H_h}(r) \ dr,
	\end{align*}
	with $D(n,q,\gamma)$ as in Theorem \ref{ThmGeneral} and $J_{H_h}$ given by \eqref{eq:JHr}. Since
	\begin{align*}
		\frac{\partial}{\partial h}
		J_{H_h}(r)=
		\begin{cases}
			0 &: 0\le r < h
			\\
			r^{-q-1} h^q \Big( 1-\frac{h^2}{r^2} \Big) ^{\frac{n-q}{2}-1} &: r>h, 
		\end{cases}
	\end{align*}
	we conclude that $F(h)$ is differentiable for almost all $h>0$
    with derivative
	\begin{equation*}
		F'(h)
		=
		D(n,q,\gamma)
	\omega_{n-\gamma} 
		\int_h^\infty 
		f_I(r) r^{-\gamma-2} h^q \Big( 1-\frac{h^2}{r^2} \Big) ^{\frac{n-q}{2}-1}
		\ dr.
	\end{equation*}
    Here, differentiation under the integral sign can be justified as follows.
For fixed $r>0$, the integrand $g(r,h)=f_I(r)r^{q-\gamma-1} J_{H_h}(r)$ is absolutely continuous as function of $h$. Since $(r,h)\mapsto \partial /(\partial h) g(r,h)$ is integrable on $(0,\infty)\times (a,b)$ for any $0<a<b<\infty$, Fubini's theorem and Lebesgue differentiation theorem imply 
\[
F'(h)=\frac{d}{dh}\int_1^h \int_0^\infty \frac{\partial}{\partial h}  g(r,t)drdt=\int_0^\infty \frac{\partial}{\partial h}  g(r,h)dr
\]
for a.e.~$h\in (a,b)$, yielding the intermediate assertion. 
    
     On the other hand,  Theorem \ref{ThmGeneral}, relation \eqref{eq:DefMuq} and spherical integration yield
	\begin{equation*}
		F(h)
		=
		\int_{G(n,q)}   \int_{G(n,n-q+\gamma)} \int_{S^{n-1} \cap M^\perp }
		\int_0^h 
		f((M+ru) \cap L)
		r^{q-\gamma-1}
		\ dr \ \mathcal{H}^{q-\gamma-1}(du)
		\ \nu_{n-q+\gamma}(dM) \nu_q (dL),
	\end{equation*}
	and hence
	\begin{equation*}
		F'(h)
		=
		\int_{G(n,q)}   \int_{G(n,n-q+\gamma)} \int_{S^{n-1} \cap M^\perp }
		f((M+hu) \cap L)
		h^{q-\gamma-1}
		\ \mathcal{H}^{q-\gamma-1}(dy)
		\ \nu_{n-q+\gamma}(dM) \nu_q (dL)
	\end{equation*}
    for almost every $h>0$ due to the Lebesgue differentiation theorem.
	This completes the proof of the lemma.
\end{proof}

		\begin{theorem}\label{thm:Application2}
			Fix $n\geq 2$, $q\in\{1,\ldots,n-1\}$ and $\gamma\in\{0,\ldots,q-1\}$. Let $L$ be a random linear subspace with distribution $\nu_q$ and let $E$ be a stochastically independent random affine subspace, tangent to the unit sphere, with distribution $\sigma_{n-q+\gamma}$ given by \eqref{eq:randomDensFlatSphere}.
			Then the random variable $d(o,E\cap L)^{-2}$ has a beta distribution with shape parameters 
			 \[
			 a=\frac{\gamma+1}2 \qquad\text{ and  }\qquad  b=\frac{n-q}2.
			 \]
			 More explicitly, 
			\begin{align*}
				g_{n,q,\gamma}(r) =   {\omega_{\gamma+1}\omega_{n-q}\over \omega_{n-(q-\gamma)+1}}\,
				r^{-(\gamma+2)}  \Big( 1-\frac{1}{r^2} \Big) ^{\frac{n-q}{2}-1}\,1_{\{r>1\}}
			\end{align*}
			is 	a  probability density for 
			$d(o,E\cap L)$. 
		\end{theorem}
	Figure \ref{fig:Gdensities} illustrates the variety of density functions $g_{n,q,\gamma}$. 
\begin{figure}[b]
	\centering
	\includegraphics[width=0.3\columnwidth]{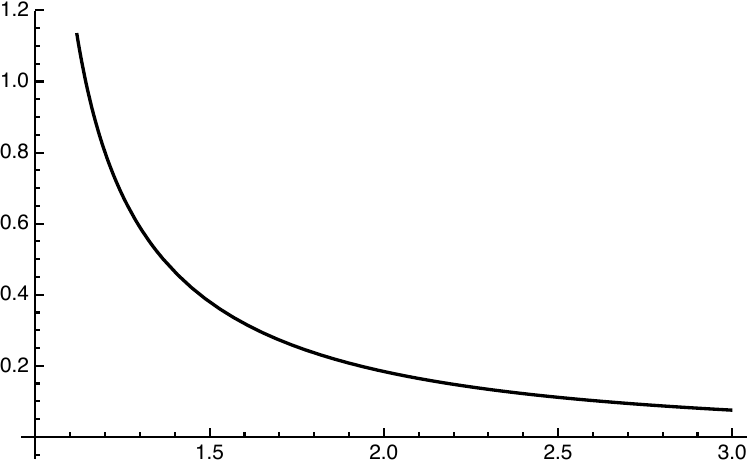}\qquad
	\includegraphics[width=0.3\columnwidth]{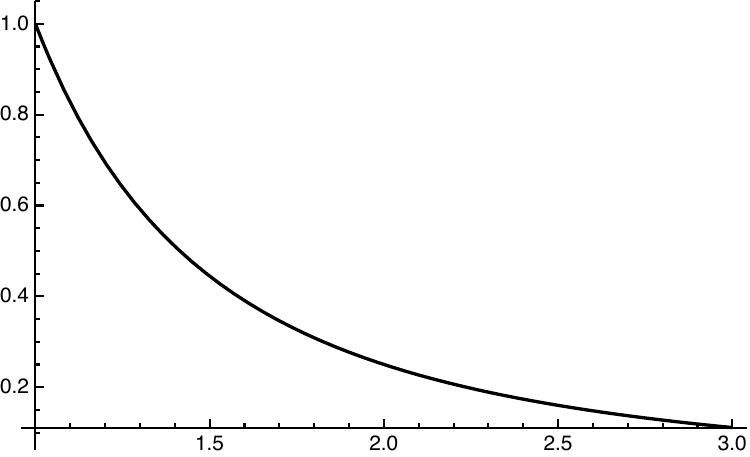}\qquad
	\includegraphics[width=0.3\columnwidth]{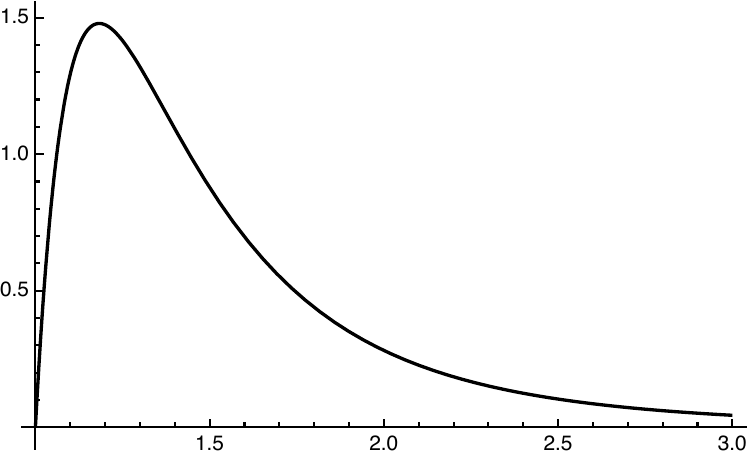}\qquad
	\caption{From left to right: The probability densities $g_{2,1,0}$, $g_{3,1,0}$ and $g_{9,5,3}$.}
	\label{fig:Gdensities}
\end{figure}
	
		\begin{proof}
	     Let $X$ be a beta-distributed random variable on $(0,1)$  with shape parameters $a,b>0$. Then $X^{-1/2}$ has tail probabilities 
	     \begin{align}\nonumber
	     1-F_{X^{-1/2}}(t)&= \mathbb{P}\Big[X\leq{1\over t^2}\Big] 
	     \\\nonumber
	     &= \frac1{B(a,b)}\int_{0}^{1/t^2}x^{a-1}(1-x)^{b-1}\ dx 
	     \\\label{eq:betadist}
	     &= \frac2{B(a,b)}\int_{t}^{\infty}r^{-2a-1}\Big(1-\frac1{r^2}\Big)^{b-1}\ dr, \qquad t>1. 
	     \end{align}
	     
		On the other hand, applying Lemma \ref{cor:Sphere} with $f(E)= 1_{\{d(o,E) \geq t \}}$ yields 
			\begin{align}\nonumber 
				\frac{1}{\omega_{q-\gamma}}    \int_{G(n,q)}   &\int_{G(n,n-q+\gamma)} \int_{S^{n-1} \cap M^\perp }
				1_{\{d(o,(M+hu) \cap L) \geq t\}}
				\ \mathcal{H}^{q-\gamma-1}(du)
				\ \nu_{n-q+\gamma}(dM) \nu_q (dL)
				\\&
				=
				D(n,q,\gamma)\frac{\omega_{n-\gamma} }
				{
					\omega_{q-\gamma}
				}h^{\gamma+1}
				\int_h^\infty 
				1_{\{r \geq t \}} r^{-(\gamma+2)}  \Big( 1-\frac{h^2}{r^2} \Big) ^{\frac{n-q}{2}-1}
				\ dr\label{eq:onemore}
			\end{align}
			for almost all $h>0$.
			Using monotone convergence, one can show that both sides of \eqref{eq:onemore} are  continuous in $h$, so this relation extends to all $h>0$. Letting $h=1$, we see that the tail distributions of $d(o,E \cap L)$ are proportional to 
			\[
			\int_{t}^\infty r^{-(\gamma+2)}  \Big( 1-\frac{1}{r^2} \Big) ^{\frac{n-q}{2}-1}
			\ dr, \qquad t>1. 
			\]
			A comparison with \eqref{eq:betadist} shows the first claim. The explicit density $g_{n,q,\gamma}$ is obtained by differentiation
			of \eqref{eq:onemore} when $h=1$.  
		\end{proof}
		 Let us remark in this context that random affine subspaces with the property that their squared distance to the origin follows a beta distribution arise naturally in the theory of random beta polytopes. For example, let $X_0,\ldots,X_r$ with $r\in\{1,\ldots,n\}$ be stochastically independent random points in $B^n$ with probability density proportional to $(1-\|x\|^2)^{\nu-2\over 2}$ for some $\nu>0$, and consider the random variable $d(o,M)^2$, where $M\in A(n,r)$ is the affine hull of $X_0,\ldots,X_r$. Then it follows from \cite[Thm.~2.7]{GroteKabluchkoThaele} that this random variable has the beta distribution $\operatorname{Beta}({n-r\over 2},{\nu(r+1)+r(n-1)\over 2})$.

	From the explicit density of $d(o,E\cap L)$ in Theorem \ref{thm:Application2}, the existence of moments can directly be read off. 
		
		\begin{corollary}
			Let $\alpha \in \mathbb{R}$. In the set-up just introduced, one has that $\mathbb{E}d(o, E \cap L)^\alpha< \infty$ if and only if $\alpha < \gamma+1$. 
			In particular, if $\gamma=0$ then the random variable $d(o, E \cap L)$ has infinite expectation. If $\gamma\ge 1$ we have 
			$$
			\mathbb{E}d(o,E\cap L) = {\omega_{\gamma+1}\omega_{n-q+\gamma}\over\omega_{\gamma}\omega_{n-(q-\gamma)+1}}
			=
			\frac{\omega_{n-q+\gamma}}{\omega_\gamma}
			(2 \pi)^{-(n-q)}. 
			$$
		\end{corollary}
        Applying Stirling's formula \cite[\S 6.1.37]{AbramowitzStegun} we conclude that for fixed $\gamma\ge 1$,
		$$
		\frac{\mathbb{E}d(o,E\cap L)}{\sqrt{n}} \longrightarrow  {1\over\sqrt{2\pi}}{\omega_{\gamma+1}\over\omega_\gamma},
		$$
		as $n\to\infty$, independently of $q$.
	
\section{Proofs of Theorem \ref{ThmGeneral},  Corollary \ref{cor:MultipleIntersections}
and Theorem \ref{ThmGeneral1}}\label{sec:Proof}

\begin{proof}[Proof of Theorem \ref{ThmGeneral}]
Given $E\in A(n,q)$, $q\in\{0,\ldots,n\}$, we will write $\lin(E)$ for the linear subspace in $G(n,q)$ parallel to $E$. In other words, $\lin(E)=E-x$ for all $x\in E$. 

For the proof of Theorem \ref{ThmGeneral}, fix  $n\geq 1$, $q\in\{1,\ldots,n-1\}$ and $\gamma\in\{0,\ldots,q-1\}$, and let $f: A(n, \gamma)\to [0,\infty)$ be a measurable function. Due to Lemma \ref{Lemma:EnoughRotational}, we may assume without loss of generality that $f$ is rotation invariant. The integral of interest is
\begin{equation*}
	I=
	\int_{G(n,q)}\int_{A(n,n-q+\gamma)} f(E \cap L) H(E)  \ \mu_{n-q+\gamma}(dE) \ \nu_{q}(dL).
\end{equation*}
 First note that $I$ is indeed well-defined as for $\nu_q$-almost all $L \in G(n,q)$ and $\mu_{n-q+\gamma}$-almost all $E \in A(n,n-q+\gamma)$ we have $\dim (E \cap L)=\gamma$. This follows from 
 \cite[Lem.~13.2.1]{SchWeil2008}, applied 
 to $L$ and $L'=\lin(E)$. 
 
 Define the function $g: A(n,q) \times A(n,n-q+\gamma) \to \mathbb{R}$ by
\begin{align*}
	g(E_1,E_2) =  f\big(\lin (E_1) \cap E_2\big) 1_{\{d(o,E_1)\leq 1 \}}H(E_2),
\end{align*}
where we let $f(\lin (E_1) \cap E_2)=0$ if $\dim (\lin (E_1) \cap  E_2)\ne \gamma$. Definition \eqref{eq:DefMuq} implies 
\begin{equation*}
	I=\frac{1}{\kappa_{n-q}}
	\int_{A(n,q)}\int_{A(n,n-q+\gamma)} g(E_1,E_2) \, \mu_{n-q+\gamma}(dE_2) \mu_{q}(dE_1).
\end{equation*}
Applying \cite[Thm.~7.2.8]{SchWeil2008} with $s_1=q$ and $s_2 = n-q+\gamma$ yields 
\begin{align}
	I
	=
	\frac{\Bar{b}}{\kappa_{n-q} }
	\int_{A(n,\gamma)} 
	\int_{A(E,q)} 
	\int_{A(E,n-q+\gamma)}
	& g(E_1,E_2) [\lin(E_1), \lin(E_2)]^{\gamma+1} \nonumber
	\\& \times \ \mu_{n-q+\gamma}^E (dE_2)  \ \mu_{q}^E(dE_1) \  \mu_\gamma(dE),  
	\label{eq:dummy1}
\end{align}
with the constant $\Bar{b}$ given by
\begin{equation}
	\Bar{b}
	=
	b_{n,n-\gamma} \frac{
		b_{n-\gamma,n-q}
		b_{n-\gamma, q-\gamma}
	}{ 
		b_{n,n-q}
		b_{n,q-\gamma}},\qquad b_{i,j} = {\omega_{i-j+1}\cdots\omega_{i}\over\omega_1\cdots\omega_j},\ i\in \N,j\in\{1,\ldots,i\}.
  \label{eq:ovB}
\end{equation}
Expanding \eqref{eq:dummy1} by decomposing the measure $\mu_{\gamma}$ according to \eqref{eq:DefMuq},
applying \cite[Eq.~(13.14)]{SchWeil2008} and Tonelli's theorem yields 
\begin{equation} 
	I
	=
	\frac{\Bar{b}}{\kappa_{n-q} }
	\int_{G(n, \gamma)}
	I_1(L_0,f) \ \nu_\gamma(dL_0) \label{eq:dummy1.5}
\end{equation}
with
\begin{align}
	I_1(L_0,f) \nonumber
	& =
	\int_{G(L_0,q)}
	\int_{G(L_0,n-q+\gamma)} 
	\int_{L_0^\perp }
	f((M+t)\cap L) 
   1_{\{d(o,L+t) \leq 1\}} H(M+t) \nonumber
	\\& 
	\hspace{4cm} \times \ \lambda_{L_0^\perp}(dt)  [M, L]^{\gamma+1}\ \nu_{n-q+\gamma}^{L_0} (dM)  
	\ \nu_{q}^{L_0}(dL). \nonumber
\end{align}
Since
\begin{equation*}
	\int_{G(L_0,p)} h(L) \ \nu_p^{L_0}(dL)
	=
	\int_{G(L_0^\perp, p-\gamma)} h(L_0+L) \ \nu_{p-\gamma}^{L_0^\perp}(dL)
\end{equation*}
for all $L_0 \in G(n, \gamma)$ and for any measurable function $h:G(L_0,p)\to[0,\infty)$ and integers $\gamma < p <n$, we conclude 
\begin{equation}
	\begin{split}
		I_1(L_0,f) & = 
		\int_{G(L_0^\perp,q-\gamma)} 
		\int_{G(L_0^\perp,n-q)} 
		\int_{L_0^\perp }
		f\Big(\big((M+t)\cap L\big)+ L_0 \Big) 1_{\{d(o,L+t) \leq 1\}} \\& 
		\qquad \qquad \times H(M+L_0+t) \ \lambda_{L_0^\perp}(dt) 
  [M, L]^{\gamma+1}
  \ \nu_{n-q}^{L_0^\perp} (dM)  
		\ \nu_{q-\gamma}^{L_0^\perp}(dL),    \label{eq:dummy2}
	\end{split}
\end{equation}
where we also have used $(M+L_0+t) \cap (L+L_0) = \big( (M+t) \cap L \big)+L_0$,  $d(o,L+L_0+t)=d(o,L+t)$ and $[M+L_0, L+L_0]= [M,L]$.

Now consider \eqref{eq:dummy2} for fixed  $L_0 \in G(n, \gamma)$, $M \in G(L_0^\perp, n-q)$ and $L \in G(L_0^\perp, q-\gamma)$. 
As $f$ and $H$ are both rotation invariant, we may write
\begin{align*}
	f\big(((M+t)\cap L)+ L_0\big) 
	&=
	f_I\big( d(o, (M+t) \cap L) \big),
	\\
	H(M+L_0+t)&= H_I(\|t|M^\perp \|),
\end{align*}
with $f_I$ and $H_I$ satisfying \eqref{eq:fI}. Thus, we may identify $L_0^\perp$ with $\mathbb{R}^{n-\gamma}$ and conclude that
\begin{align}
	I_1(L_0,f)
	&=
	\int_{G(n-\gamma,q-\gamma)} 
	\int_{G(n-\gamma,n-q)} I_2(L,M)
	[M,L]^{\gamma+1}  \nu_{n-q} (dM) 
	\ \nu_{q-\gamma}(dL),
 \label{eq:I1in_Rn-q}
\end{align}
%Reinserting this   into %\eqref{eq:dummy1.5}, we arrive at
%\begin{align}
%	I
%	=
%	\frac{\Bar{b}}{\kappa_{n-q} }
%	\int_{G(n-\gamma,q-\gamma)} 
%\int_{G(n-\gamma,n-q)} %%
%[M,L]^{\gamma+1}
%	g(M,L)
%	\ \nu_{n-q} (dM)  
%	\ \nu_{q-\gamma}(dL)
%\end{align}
with
\begin{align*}
	I_2(M,L)
	&=
	\int_{\mathbb{R}^{n-\gamma}}
	f_I(d(o,(M+t)\cap L))    1_{\{d(o,L+t) \leq 1\}} H_I(\|t|M^\perp \|)
	\ \lambda_{ {n-\gamma}}(dt)
	\\
	&=
	\int_{M} \int_{M^\perp} 
	f_I( d(o,(M+x)\cap L)) 
	H_I(\|x \|)  1_{\{d(o,L+x+y) \leq 1\}} \ 
	\ \lambda_{M^\perp}(dx) \ \lambda_M(dy) 
\end{align*}
for $M \in G(n-\gamma, n-q)$ and $L \in G(n-\gamma,q-\gamma)$. Applying \eqref{eq:IntroProjectionFormula} to the Lebesgue integral over $M^\perp$ we get
\begin{align*}
	I_2(M,L) &= [M,L] 
	\int_{M} \int_{L}
	f_I( d(o,(M+t)\cap L)) 
	H_I(\|t|M^\perp  \|)  1_{\{d(o,L+t|M^\perp+y) \leq 1\}} \ 
	\ \lambda_{L}(dt) \ \lambda_M(dy) 
	\\
	&
	=   \int_{L}  f_I( d(o,(M+t)\cap L))  H_I(\|t|M^\perp  \|) I_3(M,L,t) \ \lambda_{L}(dt),
\end{align*}
where we first used the fact that 
$M+t|M^\perp=M+t$ holds for any $t \in L$, and then Tonelli's theorem.  
Here, 
\[
I_3(M,L,t) =[M,L]
 \int_{M}  
	1_{\{\|(t| M^\perp)|L^\perp+y|L^\perp \| \leq 1\}} \ 
\lambda_M(dy).
\]
Another application of  \eqref{eq:IntroProjectionFormula} 
reveals that 
\[
I_3(M,L,t) = \int_{L^\perp}  
	1_{\{\|(t| M^\perp)|L^\perp+z \| \leq 1\}} \ 
\lambda_{L^\perp}(dz)=\kappa_{n-q}
\]
is the volume of a unit ball in $L^\perp$, centered at $(t| M^\perp)|L^\perp$. Inserting this into $I_2$ gives 
\begin{equation*}
	I_2(M,L)
	= \kappa_{n-q} \int_{L}  f_I( d(o,(M+t)\cap L))  H_I(\|t|M^\perp  \|)  \ \lambda_{L}(dt)  .  
\end{equation*}
Since $t\in L$ and  $M\cap L=\{o\}$ for almost all $L$ and $M$, we have 
\[
d(o,(M+t) \cap L)=
d(o,\{t\})=\|t\|.
\]
This, and the use of 
spherical coordinates in $L$ give
\begin{equation*}
	I_2(M,L)
	=\kappa_{n-q}
	\int_0^\infty f_I(r) r^{q-\gamma-1}
	\int_{S^{n-\gamma-1}\cap L}
	H_I(r [u,M])
	\ \mathcal{H}^{q-\gamma-1}(du) \ dr.
\end{equation*}
Inserting this into \eqref{eq:I1in_Rn-q} 
and the result into \eqref{eq:dummy1.5}, 
we get after an application of Tonelli's theorem that 
\begin{align}\label{eq:IJM}
	I
	&=\bar{b}
	\int_0^\infty f_I(r) r^{q-\gamma-1} \int_{G(n-\gamma,n-q)} J_H(M,r)
	\ \nu_{n-q} (dM)
	\ dr,
\end{align}
with 
\begin{align*}
	J_H(M,r)=  
	\int_{G(n-\gamma,q-\gamma)} 
	\int_{S^{n-\gamma-1}\cap L} 
	H_I(r [u,M])\mathcal{H}^{q-\gamma-1}(du) 
	%\\& \hspace{1cm}\times \   
   [M,L]^{\gamma+1}
	\ \nu_{q-\gamma}(dL). 
\end{align*}
An invariance argument and  \cite[Thm.~7.1.1]{SchWeil2008} imply 
\begin{align*}
	J_H(M,r)=\frac{\omega_{q-\gamma}}{\omega_{n-\gamma}}  
	\int_{S^{n-\gamma-1}} 
	H_I(r [u,M])
	\int_{G(\myspan u, q-\gamma)}
	[M,L]^{\gamma+1}
	%\\& \hspace{7cm}\times 
	\ \nu_{q-\gamma}^{\myspan u}(dL)
	\  \mathcal{H}^{n-\gamma-1}(du) . 
\end{align*}
 Applying Lemma \ref{Lemma:IntegralSubspaceDet} in $\R^{n-\gamma}$ to the innermost integral in $J_H(M,r)$ yields
\begin{equation*}
	J_H(M,r)
	=
	c_1\int_{S^{n-\gamma-1}}
	H_I(r[u,M])
	[u,M]^{\gamma+1}
	\ \mathcal{H}^{n-\gamma-1}(du)
\end{equation*}
with the constant 
$c_1=a(n-\gamma,n-q,q- \gamma,\gamma+1)\frac{\omega_{q-\gamma}}{\omega_{n-\gamma}}$. 
To simplify $J_H(M,r)$ further, we use the fact that  $[u,M]=\|u|M^\perp\|$ and apply \cite[Lem.~1]{Auneau2010} with $B_p= M^\perp$, $p=q-\gamma$ and $d= n-\gamma$. This yields
\begin{align*}
	J_H(M,r)
	&
	= \frac{c_1}{2} \omega_{n-q} \omega_{q-\gamma}
	\int_0^1
	H_I(r t^{\frac{1}{2}}) t^{ \frac{q-1}{2}}(1-t)^{ \frac{n-q}{2}-1}
	\ dt 
	\\
	&=
	c_1\omega_{n-q} \omega_{q-\gamma}
	\int_0^1
	H_I(r z) z^{ q}(1-z^2)^{ \frac{n-q}{2}-1}
	\ dz, 
\end{align*}
using the substitution $z= \sqrt{t}$ in the last step. A comparison with \eqref{eq:defJ} gives $J_H(M,r)=c_1\omega_{n-q} \omega_{q-\gamma} J_H(r)$, so abbreviating 
\[
 c_2=\bar{b}c_1\frac{\omega_{n-q} \omega_{q-\gamma}}{\omega_{n-\gamma}},
\]
\eqref{eq:IJM} becomes 
\begin{align*}
	I
	&=c_2\omega_{n-\gamma}
	\int_0^\infty f_I(r) r^{q-\gamma-1}J_H(r)\ dr
 \\
 &= c_2\int_{A(n, \gamma)}
	f( E) d(o,E)^{-(n-q)} J_H(d(o,E)) \ \mu_\gamma(dE),
\end{align*}
where we used \eqref{eq:DefMuq} and spherical coordinates in $\lin(E)$. 
Hence, the theorem is shown once we have confirmed that 
\begin{equation} 
c_2=D(n,q,\gamma). 
\label{eq:!!!const}
\end{equation}
We have 
			 \[			 c_1
    =\frac{\omega_{q-\gamma}}{\omega_{n-\gamma}}
  \prod_{i=1}^{n-q}
	\frac{\Gamma(\frac{n-\gamma-i}{2})
		\Gamma(\frac{n-q-i+\gamma}{2}+1)
	}
	{
		\Gamma( \frac{n-i+1}{2})
		\Gamma( \frac{n-q-i+1}{2})
	}
	=
    \frac{\omega_{q-\gamma}}{\omega_{n-\gamma}}
			 \prod_{i=1}^{n-q}
			 	\frac{
				 		\omega_{n-i+1}
				 		\omega_{n-q-i+1}
				 	}
			 	{
				 		\omega_{n-\gamma-i}
				 		\omega_{n-q+\gamma+2-i}
				 	},
			 \]
% \begin{equation*}
% 	C_2 = \bar{b}\,C_1\, \frac{\omega_{q-\gamma}}{\omega_{n-\gamma}}
% 	\omega_{n-q} \omega_{q-\gamma}.
% \end{equation*}
% To complete the proof, we need to simplify the constant $C_2$. 
so direct insertion gives
\begin{align}
	\bar{b}c_1
	&= \nonumber
	\Bigg(
	\frac{%\color{red}
    \omega_{\gamma+1}\cdots \omega_n %\color{black}    
	}
	{
		%\color{blue}
        \omega_1 \cdots \omega_{n-\gamma}
  %\color{black}
	}
	\cdot
	\frac{
		%\color{green}
  \omega_{q-\gamma+1} \cdots \omega_{n-\gamma}
   %\color{black}	
}
	{
		\omega_{q+1}\cdots \omega_{n}
	}
	\cdot
	\frac{
%\color{blue}
\omega_{n-q+1}\cdots \omega_{n-\gamma}
	%\color{black}
 }
	{
%\color{brown}		
\omega_{n-(q-\gamma)+1}\cdots \omega_{n}
	%\color{black}
    }
	\Bigg)
 \\
	&\nonumber\hspace{4cm}\times
	\Bigg(
	\frac{
		\omega_{q+1}\cdots \omega_n
	}
	{
    %\color{green}
    \omega_{q-\gamma}
		\cdots \omega_{n-\gamma-1}
    %\color{black}
}
	\cdot
	\frac{
		%\color{blue}
     \omega_{1}\cdots \omega_{n-q}
    %\color{black}
}
	{
		%\color{brown}
  \omega_{\gamma+2}
		\cdots
		\omega_{n-(q-\gamma)+1}
 %\color{black}
 }
	\Bigg)\frac{\omega_{q-\gamma}}{\omega_{n-\gamma}}
	\\
 &= \nonumber
	\Bigg(
	\frac{\omega_{q+1}\cdots \omega_{n}
	}
	{
		\omega_{q+1}\cdots \omega_{n}
	}
	\cdot
	\frac{
		%\color{red}
        \omega_{\gamma+1}\cdots \omega_n %\color{black}
}
	{
%\color{brown}		
\omega_{\gamma+2}\cdots \omega_{n}	
%\color{black}
	}
 \cdot
 \frac{1}{%\color{brown} 
 \omega_{n-(q-\gamma)+1}  
 %\color{black}
 }
	\Bigg)\\
	&\nonumber\hspace{4cm}\times
	\Bigg(
 \frac{
 %\color{green}
  \omega_{q-\gamma} \cdots \omega_{n-\gamma}
%\color{black}	
 }{%\color{green}		
 \omega_{q-\gamma}
		\cdots \omega_{n-\gamma}}
	\cdot\frac{%\color{blue}
 \omega_{1} \cdots \omega_{n-\gamma} %\color{black}
	}
	{
 %\color{blue}
 \omega_{1} \cdots \omega_{n-\gamma}
%\color{black}
}
	\Bigg)
 \\
 &=\nonumber
 \frac{\omega_{\gamma+1}}
 {\omega_{n-(q-\gamma)+1}},
\end{align}
where the products were suitably sorted at the second equality sign, and simplified at the third equality sign. We thus get 
\begin{equation*}
	c_2=\frac{\omega_{\gamma+1} \omega_{q-\gamma}\omega_{n-q}}{\omega_{n-(q-\gamma)-1}\omega_{n-\gamma}} =D(n,q,\gamma).
\end{equation*}
This shows \eqref{eq:!!!const} and completes the proof.
\end{proof}

\begin{proof}[Proof of Corollary \ref{cor:MultipleIntersections}]
For a Borel set $B\subseteq G(n,q)$,  define 
\[
\tilde \nu_q(B)=
\int_{G(n,q_1)}\cdots\int_{G(n,q_\ell)}1_B(L_1\cap\ldots\cap L_\ell)\,\nu_{q_\ell}(dL_\ell)\ldots\nu_{q_1}(dL_1).
\]
This gives rise to an invariant probability measure $\tilde \nu_q$ on $G(n,q)$.  However, by \cite[Thm.~13.2.11]{SchWeil2008} there is only one such measure, which implies that $\tilde \nu_q=\nu_q$.
This allows us in a first step to reduce the outer integral over $G(n,q_1),\ldots,G(n,q_m)$ in $I_{\ell,m}$ to a single integral over $G(n,q)$:
\begin{equation}\label{eq:IlmFirstReduction}
\begin{split}
I_{\ell,m} &= \int_{G(n,q)}\int_{A(n,p_1)}\cdots\int_{A(n,p_m)}f(E_1\cap\ldots\cap E_m\cap L)\\
&\qquad\qquad\qquad\qquad\times H(E_1\cap\ldots\cap E_m)\,\mu_{p_m}(dE_m)\ldots\mu_{p_1}(dE_1)\nu_q(dL).    
\end{split}
\end{equation}
To also convert the inner integrals over $A(n,p_1),\ldots,A(n,p_m)$ to a single integral, let $B$ be a Borel set in $A(n,n-q+\gamma)$ and define
\begin{align*}
    \widehat{\mu}_{n-q+\gamma}(B) = \int_{A(n,p_1)}\cdots\int_{A(n,p_m)}1_B(E_1\cap\ldots\cap E_m)\,\mu_{p_m}(dE_m)\ldots\mu_{p_1}(dE_1).
\end{align*}
The measure $\widehat{\mu}_{n-q+\gamma}$ is motion invariant on $A(n,n-q+\gamma)$.  However, by \cite[Thm.~13.1.3 and Thm.~13.2.12]{SchWeil2008} all such measures are constant multiples of the invariant measure 
$\mu_{n-q+\gamma}$, so there exists some constant $c>0$ such that $\widehat{\mu}_{n-q+\gamma}=c\mu_{n-q+\gamma}$. 
To determine the value of the constant $c$, we employ a special case of the Crofton formula \cite[Thm.~5.1.1]{SchWeil2008}. It says that
\begin{align*}
    \int_{A(n,k)}\mathcal{H}^i(E\cap W)\,\mu_k(dE) = {\omega_{n+1}\omega_{i+1}\over\omega_{k+1}\omega_{n-k+i+1}}\mathcal{H}^{n-k+i}(W)
\end{align*}
for $0\leq i\leq k\leq n-1$ and where $W\subset\R^n$ is a convex set of dimension $n-k+i$, that is, the affine hull of $W$ has dimension $n-k+i$.
 Applying Crofton's formula with $i=k=n-q+\gamma$ and $W=B^n$ we obtain 
\begin{align*} \int_{A(n,n-q+\gamma)}\mathcal{H}^{n-q+\gamma}(E\cap B^n)\,\widehat{\mu}_{n-q+\gamma}(dE) &= c\int_{A(n,n-q+\gamma)}\mathcal{H}^{n-q+\gamma}(E\cap B^n)\,\mu_{n-q+\gamma}(dE) \\&= c\mathcal{H}^n(B^n) = c\kappa_n.
\end{align*}
On the other hand, by applying Crofton's formula repeatedly to each of the integrals 
in its definition, the integral on the left is
\begin{align*}
   & \int_{A(n,p_1)}\cdots\int_{A(n,p_m)}\mathcal{H}^{n-q+\gamma}(E_1 \cap \ldots \cap E_m\cap B^n)\,\mu_{p_m}(dE_m)\ldots\mu_{p_1}(dE_1)\\
   &={\omega_{n+1}\omega_{n-q+\gamma+1}\over\omega_{p_m+1}\omega_{2n-p_m-q+\gamma+1}}\int_{A(n,p_1)}\cdots\int_{A(n,p_{m-1})}\mathcal{H}^{2n-p_m-q+\gamma}(E_1\cap\ldots\cap E_{m-1}\cap B^n)\\
   &\hspace{8cm}\mu_{p_{m-1}}(dE_{m-1})\ldots\mu_{p_1}(dE_1)\\
   &\vdots\\
   &={\omega_{n+1}\omega_{n-q+\gamma+1}\over\omega_{p_m+1}\omega_{2n-p_m-q+\gamma+1}}{\omega_{n+1}\omega_{2n-p_m-q+\gamma}\over\omega_{p_{m-1}}\omega_{3n-p_m-p_{m-1}-q+\gamma+1}}\cdots\\
   &\hspace{6cm}\cdots {\omega_{n+1}\omega_{mn-p_m-\ldots-p_2-q+\gamma}\over\omega_{p_1+1}\omega_{n-p_1+mn-p_m-\ldots-p_2-q+\gamma+1}}\,\mathcal{H}^n(B^n)\\
   &={\omega_{n+1}^m\omega_{n-q+\gamma+1}\over\omega_{p_1}\cdots\omega_{p_m}\omega_{n+1}}\,\kappa_n.
\end{align*}
Here, we simplified the telescopic product of the $\omega$-terms and used our assumption $p_1+\ldots+p_m-(m-1)n=n-1+\gamma$ in the last step. A comparison of these two expressions implies 
$$
\widehat{\mu}_{n-q+\gamma} = {\omega_{n+1}^m\omega_{n-q+\gamma+1}\over\omega_{p_1}\cdots\omega_{p_m}\omega_{n+1}}\,\mu_{n-q+\gamma}.
$$
As a consequence, we can reduce the inner integrals in \eqref{eq:IlmFirstReduction} to a single integral over $A(n,n-q+\gamma)$:
\begin{align*}
    I_{\ell,m} = {\omega_{n+1}^m\omega_{n-q+\gamma+1}\over\omega_{p_1}\cdots\omega_{p_m}\omega_{n+1}}\int_{G(n,q)}\int_{A(n,n-q+\gamma)}f(E\cap L)H(E)\,\mu_{n-q+\gamma}(dE)\nu_q(dL).
\end{align*}
The result of Corollary \ref{cor:MultipleIntersections} can now be concluded from Theorem \ref{ThmGeneral}.
\end{proof}

The proof that the results of Theorems \ref{ThmGeneral} and \ref{ThmGeneral1} are equivalent, is based on the following Blaschke--Petkantschin formula: 
\begin{equation}\label{eq:BP}
    \int_{A(n,k)} g(E)\mu_k(dE)=\frac{\omega_{n-k}}{\omega_{r-k}}\int_{G(n,r)}
    \int_{A(L,k)} g(E)d(o,E)^{n-r}\mu_k^{L}(dE)\nu_r(dL), 
\end{equation}
with integers $0\le k<r<n$ and measurable $g:A(n,k)\to [0,\infty)$, see e.g.~\cite[p.~33]{KiderlenSurvey}. 

\begin{proof}[Proof of the equivalence of Theorems \ref{ThmGeneral} and \ref{ThmGeneral1}]
That Theorem \ref{ThmGeneral1} implies Theorem \ref{ThmGeneral}
follows ra\-ther directly by invariant integration of  \eqref{eq:new} with respect to $L_0$ and an application of the Blaschke--Petkantschin formula \eqref{eq:BP} with $r=q$ and $k=\gamma$. 

Assume now that Theorem \ref{ThmGeneral} holds, and let $f:A(L_0,\gamma)\to[0,\infty)$ be measurable. Assume first that 
$f$ is rotation invariant under all rotations fixing $L_0$. The 
function $\tilde f(E)=f_I(d(o,E))$, $E\in A(n,\gamma)$, is the rotation invariant extension of $f$ to $A(n,\gamma)$. 
Hence, the invariance of $\mu_{n-q+\gamma}$ gives
\begin{align*}
    &\int_{G(n,q)}
		\int_{A(n,n-q+\gamma)}
		\tilde f(E \cap L) H(E) \ \mu_{n-q+\gamma}(dE) \ \nu_q(dL)
  \\
  &=\int_{SO(n)}
		\int_{A(n,n-q+\gamma)}
		\tilde f(E \cap \vartheta L_0) H(E) \ \mu_{n-q+\gamma}(dE) \ \nu(d\vartheta)
  \\
  &=\int_{SO(n)}
		\int_{A(n,n-q+\gamma)}
		\tilde f\big(\vartheta(E \cap L_0)\big) H(\vartheta^{-1}E) \ \mu_{n-q+\gamma}(dE) \ \nu(d\vartheta)
   \\
   &=	\int_{A(n,n-q+\gamma)}
		f(E \cap L_0)
        H(E) \ \mu_{n-q+\gamma}(dE). 
\end{align*}
On the other hand, the Blaschke--Petkantschin relation \eqref{eq:BP} and a similar reasoning shows 
\[
\int_{A(n, \gamma)}
\tilde f( E) d(o,E)^{-(n-q)} J_H(d(o,E)) \ \mu_\gamma(dE)=
\frac{\omega_{n-\gamma}}{\omega_{q-\gamma}}
\int_{A(L_0, \gamma)}
f(E) J_H(d(o,E)) \ \mu_\gamma^{L_0}(dE).  
\]
Theorem \ref{ThmGeneral} states that these two displayed expressions coincide up to multiplication with the constant $D(n,q,\gamma)$, and thus \eqref{eq:new} holds for the function $f$ chosen.  

With arguments as in the proof of Lemma \ref{Lemma:EnoughRotational}, one can show that \eqref{eq:new} holds for all measurable $f:A(L_0,\gamma)\to [0,\infty)$ if it holds for all such functions which are in addition invariant under all rotations fixing $L_0$. This proves that Theorem \ref{ThmGeneral1} holds and concludes the proof of the equivalence. 
\end{proof}

\subsection*{Acknowledgment}
ED has been supported by Sapere Aude: DFF-Starting Grant 0165-00061B. CT has been supported by the German Research Foundation (DFG) via SPP 2265 \textit{Random Geometric Systems} and the project \textit{Limit theorems for the volume of random projections of $\ell_p$-balls} (project number 516672205).

\bibliographystyle{plain}
\bibliography{Bib}

\begin{thebibliography}{10}

\bibitem{AbramowitzStegun}
Milton Abramowitz and Irene~A. Stegun, editors.
\newblock {\em Handbook of {M}athematical {F}unctions with {F}ormulas,
  {G}raphs, and {M}athematical {T}ables}.
\newblock Dover Publications, Inc., New York, 1992.
\newblock Reprint of the 1972 edition.

\bibitem{Auneau2010}
Jérémy Auneau and Eva B.~Vedel Jensen.
\newblock Expressing intrinsic volumes as rotational integrals.
\newblock {\em Adv. Appl. Math.}, 45(1):1--11, 2010.

\bibitem{BlaschkeIntegralgeometrie}
Wilhelm Blaschke.
\newblock {\em Vorlesungen \"{u}ber {I}ntegralgeometrie}.
\newblock Deutscher Verlag der Wissenschaften, Berlin, 1955.
\newblock 3te Aufl.

\bibitem{ChasapisEtAl}
Giorgos Chasapis, Apostolos Giannopoulos, and Dimitris-Marios Liakopoulos.
\newblock Estimates for measures of lower dimensional sections of convex
  bodies.
\newblock {\em Adv. Math.}, 306:880--904, 2017.

\bibitem{DalMasoEtAl}
Gianni Dal~Maso, Irene Fonseca, and Giovanni Leoni.
\newblock Asymptotic analysis of second order nonlocal {C}ahn-{H}illiard-type
  functionals.
\newblock {\em Trans. Amer. Math. Soc.}, 370(4):2785--2823, 2018.

\bibitem{DannEtAl}
Susanna Dann, Grigoris Paouris, and Peter Pivovarov.
\newblock Bounding marginal densities via affine isoperimetry.
\newblock {\em Proc. Lond. Math. Soc. (3)}, 113(2):140--162, 2016.

\bibitem{DareKiderlen2024}
Emil Dare and Markus Kiderlen.
\newblock Rotational {C}rofton formulae with a fixed subspace.
\newblock {\em Adv. Appl. Math.}, 153:141--191, 2024.

\bibitem{GardnerBook}
Richard~J. Gardner.
\newblock {\em Geometric tomography}, volume~58 of {\em Encyclopedia of
  Mathematics and its Applications}.
\newblock Cambridge University Press, New York, second edition, 2006.

\bibitem{GardnerJensenVolcicSurvey}
Richard~J. Gardner, Eva B.~Vedel Jensen, and Aljosa Vol\v{c}i\v{c}.
\newblock Geometric tomography and local stereology.
\newblock {\em Adv. in Appl. Math.}, 30(3):397--423, 2003.

\bibitem{GroteKabluchkoThaele}
Julian Grote, Zakhar Kabluchko, and Christoph Th\"{a}le.
\newblock Limit theorems for random simplices in high dimensions.
\newblock {\em ALEA Lat. Am. J. Probab. Math. Stat.}, 16(1):141--177, 2019.

\bibitem{HaddadLudwig}
Juli{\'a}n {Haddad} and Monika {Ludwig}.
\newblock {Affine Fractional Sobolev and Isoperimetric Inequalities}.
\newblock {\em arXiv e-prints}, page arXiv:2207.06375, July 2022.

\bibitem{HugSchneiderSchuster}
Daniel Hug, Rolf Schneider, and Ralph Schuster.
\newblock Integral geometry of tensor valuations.
\newblock {\em Adv. in Appl. Math.}, 41(4):482--509, 2008.

\bibitem{JensenBook}
Eva B.~Vedel Jensen.
\newblock {\em Local stereology}, volume~5 of {\em Advanced Series on
  Statistical Science \& Applied Probability}.
\newblock World Scientific Publishing Co., Inc., River Edge, NJ, 1998.

\bibitem{KiderlenSurvey}
Markus Kiderlen.
\newblock Introduction into integral geometry and stereology.
\newblock In {\em Stochastic geometry, spatial statistics and random fields},
  volume 2068 of {\em Lecture Notes in Math.}, pages 21--48. Springer,
  Heidelberg, 2013.

\bibitem{Koetzer}
Stephan K\"{o}tzer.
\newblock Geometric identities in stereological particle analysis.
\newblock {\em Image Anal. Stereol.}, 25(2):63--74, 2006.

\bibitem{LudwigFracPerim}
Monika Ludwig.
\newblock Anisotropic fractional perimeters.
\newblock {\em J. Differential Geom.}, 96(1):77--93, 2014.

\bibitem{MilmanYehudayoff}
Emanuel Milman and Amir Yehudayoff.
\newblock Sharp isoperimetric inequalities for affine quermassintegrals.
\newblock {\em J. Amer. Math. Soc.}, 36(4):1061--1101, 2023.

\bibitem{Petkantschin}
Bojan Petkantschin.
\newblock Integralgeometrie 6. {Z}usammenh\"{a}nge zwischen den {D}ichten der
  linearen {U}nterr\"{a}ume imn- dimensionalen {R}aum.
\newblock {\em Abh. Math. Sem. Univ. Hamburg}, 11(1):249--310, 1935.

\bibitem{Rataj1999}
Jan Rataj.
\newblock Translative and kinematic formulae for curvature measures of flat
  sections.
\newblock {\em Math. Nachr.}, 197(1):89--101, 1999.

\bibitem{RubinDrury}
Boris Rubin.
\newblock A note on the {B}laschke-{P}etkantschin formula, {R}iesz
  distributions, and {D}rury's identity.
\newblock {\em Fract. Calc. Appl. Anal.}, 21(6):1641--1650, 2018.

\bibitem{SantaloBook}
Luis~A. Santal\'{o}.
\newblock {\em Integral geometry and geometric probability}.
\newblock Cambridge Mathematical Library. Cambridge University Press,
  Cambridge, second edition, 2004.
\newblock With a foreword by Mark Kac.

\bibitem{SchWeil2008}
Rolf Schneider and Wolfgang Weil.
\newblock {\em Stochastic and Integral Geometry}.
\newblock Probability and Its Applications. Springer, Heidelberg, 2008.

\end{thebibliography}

\end{document}